\documentclass[12pt,reqno]{amsart}
\usepackage{amsmath,amsfonts,amsthm,amssymb,color}

\usepackage{hyperref}
\usepackage[T1]{fontenc}
\usepackage{pdfsync}
\usepackage[width=6.5in, left=0.95in, right=1.05in, height=8.8in, top=1.1in,bottom=1.1in]{geometry}

\newcommand{\ou}{[0,1]}
\definecolor{dg}{rgb}{0, 0.5, 0}

\newcommand{\bb}{\mathbf{B}}

\newcommand{\bch}{\bar{\mathcal{H}}}
\newcommand{\bd}{\mathbf{D}}
\newcommand{\bj}{\mathbf{J}}

\newcommand{\capa}{\text{Cap}}
\newcommand{\der}{\delta}

\newcommand{\id}{\mbox{Id}}

\newcommand{\iot}{\int_{0}^{t}}

\newcommand{\ot}{[0,t]}

\newcommand{\1}{{\bf 1}}

\newcommand{\lp}{\left(}
\newcommand{\rp}{\right)}
\newcommand{\lc}{\left[}
\newcommand{\rc}{\right]}
\newcommand{\lcl}{\left\{}
\newcommand{\rcl}{\right\}}
\newcommand{\lln}{\left|}
\newcommand{\rrn}{\right|}
\newcommand{\lla}{\left\langle}
\newcommand{\rra}{\right\rangle}

\newcommand{\ep}{\varepsilon}

\newcommand{\ga}{\gamma}

\newcommand{\la}{\lambda}

\newcommand{\om}{\omega}
\newcommand{\oom}{\Omega}

\newcommand{\si}{\sigma}

\newcommand{\vp}{\varphi}

\newcommand{\beq}{\begin{equation}}
\newcommand{\eeq}{\end{equation}}
\newcommand{\bea}{\begin{eqnarray}}
\newcommand{\eea}{\end{eqnarray}}
\newcommand{\beas}{\begin{eqnarray*}}
\newcommand{\eeas}{\end{eqnarray*}}

\def\cC{{\mathcal C}}

\def\me{{\mathbb  E}}

\def\md{{\mathbb D}}
\def\mr{{\mathbb  R}}

\def\mn{{\mathbb  N}}

\def\mp{{\mathbb  P}}

\newcommand{\cac}{{\mathcal C}}

\newcommand{\ce}{{\mathcal E}}
\newcommand{\cf}{{\mathcal F}}

\newcommand{\ch}{{\mathcal H}}

\newcommand{\cn}{{\mathcal N}}

\newcommand{\crr}{{\mathcal R}}

\newcommand{\D}{{\mathbb D}}
\newcommand{\EE}{{\mathbb E}}

\newcommand{\PP}{{\mathbb P}}

\newcommand{\R}{{\mathbb R}}

\newtheorem{theorem}{Theorem}[section]

\newtheorem{corollary}[theorem]{Corollary}

\newtheorem{definition}[theorem]{Definition}

\newtheorem{hypothesis}[theorem]{Hypothesis}
\newtheorem{lemma}[theorem]{Lemma}

\newtheorem{proposition}[theorem]{Proposition}
\theoremstyle{remark}
\newtheorem{remark}[theorem]{Remark}
\theoremstyle{remark}
\newtheorem{example}[theorem]{Example}
\theoremstyle{remark}

\newtheorem{foo}[theorem]{Remarks}

\newenvironment{Remark}{\begin{remark}\rm}{\end{remark}}

\title[Probability laws of fractional RDEs]{On probability laws of solutions to differential systems driven by a fractional Brownian motion}

\author[F. Baudoin \and E. Nualart \and C. Ouyang \and S. Tindel]{F. Baudoin \and E. Nualart \and C. Ouyang \and S. Tindel}

\thanks{First author supported in part by
NSF Grant DMS 0907326. Second author acknowledges support from the European Union program 
 FP7-PEOPLE-2012-CIG under grant agreement 333938.  Fourth author is member of the BIGS (Biology, Genetics and Statistics) team at INRIA}



\date{\today}

\address{Fabrice Baudoin, Dept. Mathematics, Purdue University, 150 N. University St., West Lafayette, IN 47907-2067, USA.}
\email{fbaudoin@math.purdue.edu}

\address{Eulalia Nualart, Dept. Economics and Business, Universitat Pompeu Fabra and Barcelona Graduate School of Economics, Ram\'on Trias Fargas 25-27, 08005 Barcelona, Spain}
\email{eulalia@nualart.es}

\address{Cheng Ouyang, Dept. Mathematics, Statistics and Computer Science, University of Illinois at Chicago, 851 S. Morgan St., Chicago, IL 60607, USA.}
\email{couyang@math.uic.edu}

\address{Samy Tindel, Institut {\'E}lie Cartan, Universit\'e de Lorraine, B.P. 239,
54506 Vand{\oe}u\-vre-l{\`e}s-Nancy, France.}
\email{Samy.Tindel@univ-lorraine.fr}

\begin{document}

\begin{abstract}
This article investigates several properties related to densities of solutions $(X_t)_{t\in\ou}$ to differential equations driven by a fractional Brownian motion with Hurst parameter $H>1/4$. We first determine conditions for strict positivity of the density of $X_t$. Then we obtain some exponential bounds for this density when the diffusion coefficient satisfies an elliptic type condition. Finally, still in the elliptic case, we derive some bounds on the hitting probabilities of sets by fractional differential systems in terms of Newtonian capacities.
\end{abstract}

\maketitle

\tableofcontents

\section{Introduction}

Let $B=(B^1,\ldots,B^d)$ be a $d$-dimensional fractional Brownian motion indexed by $\ou$, with Hurst parameter $H>1/4$, defined on a complete probability space $(\Omega,\cf,\PP)$. Recall that this means that the components $B^i$ are i.i.d and that each $B^i$ is a centered Gaussian process satisfying
\begin{equation}\label{eq:var-increments-B}
\EE\lc (B_t^i-B_s^i)^{2} \rc = |t-s|^{2H}.
\end{equation}
In particular, for any $H>1/4$, the path $t\mapsto B_t$ is almost surely $(H-\ep)$-H\"older continuous for any $\ep>0$ and for $H=1/2$ the process $B=B^{H}$ coincides with the usual $d$-dimensional Brownian motion.

\smallskip

We are concerned here with the following class of equations driven by $B$:
\begin{equation}
\label{eq:sde-intro} X^{x}_t =x +\int_0^t V_0 (X^x_s)ds+
\sum_{i=1}^d \int_0^t V_i (X^{x}_s) dB^i_s, \quad t\in\ou,
\end{equation}
where $x$ is a generic initial condition and $\{V_i;\, 0\le i\le d\}$ is a collection of smooth vector fields of $\R^n$. Owing to the fact that the family $\{B^{H}; 0<H<1\}$ is a very natural generalization of Brownian motion, this kind of system is increasingly used in applications and has also been thoroughly analyzed in the last past years at a theoretical level. 

\smallskip

Among the contributions to the study of \eqref{eq:sde-intro} which seem most relevant to our purposes let us first mention the resolution of the equation, with Young type integration methods for $H>1/2$ (cf. \cite{Za}) and rough paths techniques for $H\in(1/4,1/2)$ (see e.g \cite{FV-bk}). Then once equation  \eqref{eq:sde-intro} is solved, a natural question to address is to get some information on the law of the random variable $X^x_t$ when $t\in(0,1]$. To this respect, we have to distinguish several cases:
\begin{itemize}
\item 
When $H>1/2$ and under ellipticity assumptions on the vector fields $V_i$, existence and smoothness of the density are shown in \cite{NS,HN}. The H\"ormander's case for $H>1/2$ is treated in \cite{BH}.
\item
When $H\in(1/4,1/2)$, the integrability of the Jacobian established in \cite{CLL} immediately yields smoothness of the density in the elliptic case. The hypoelliptic case is handled in the series of papers \cite{CF,H-P,HT}, culminating by the reference \cite{CHLT} which gives a H\"omander's type criterion for a wide class of Gaussian processes including fBm with $H\in(1/4,1/2)$.
\item
Concentration results and exponential bounds on the density are treated in particular cases: gradient bounds in the  case $H>1/2$ are obtained in \cite{BO12}, and an upper bound for the density in a skew-symmetric situation is addressed in \cite{BOT}.
\end{itemize}
Let us also mention several attempts of small time asymptotics for the density of $X^x_t$, like the expansions contained in \cite{BC07,BO,NNRT}.

\smallskip

The current article should be seen as another step towards a better understanding of the law of $X^x$ as a process when the coefficients of equation \eqref{eq:sde-intro} satisfy different kind of  ellipticity conditions.

\smallskip

The following assumption will prevail until the end of the paper:
\begin{hypothesis}\label{hyp:regularity-V}
The vector fields $V_0,\ldots,V_d$ are $C_b^{\infty}(\R^n)$ (bounded together with all their derivatives).
\end{hypothesis} 

Let us now range our non-degeneracy conditions in increasing order of restrictions: the first kind of assumption is a rather mild control-type hypothesis which can be traced back to   \cite{BL} and \cite{Bis}.

\begin{hypothesis}\label{hyp:ben-arous-leandre}
Let $\ch$ be the Hilbert space related to our fBm $B$ (see the definition at Section \ref{sec:malliavin-tools}) and define a map $\Phi: \ch\to\cC(\R^n)$ such that for all $h\in\ch$, $\Phi(h)$ is defined by the ordinary differential equation
\begin{equation*}
\Phi(h)_t=x+\int_0^tV_0(\Phi(h)_s)ds+\sum_{i=1}^d\int_0^tV_i(\Phi(h)_s)d \crr h^i_s,
\end{equation*}
which is understood in the {($p$-var)} Young sense and where the isometry $\crr$ is defined by relation \eqref{eq:def-R}.  Then for any  $y\in\R^n$, there exists an element $h\in\ch$ such that $\Phi(h)_t=y$ and $\Phi(h)$ is a submersion.
\end{hypothesis}

Hypothesis \ref{hyp:ben-arous-leandre} is a variant of H\"ormander's condition, and it has been shown in~\cite{BL} that it is equivalent to the strict positivity of the density function of $X_t^x$ in case of equations driven by Brownian motions. More precisely, as pointed out in \cite[Page 28]{Bis}, Hypothesis~\ref{hyp:ben-arous-leandre} is for instance satisfied if the following  condition is met:  For every $x \in \mathbb{R}^n$ and every non vanishing $\lambda \in \mathbb{R}^d$,  the vectors $V_1(x), \cdots, V_d(x)$ and $[V_1, Y](x), \ldots, [V_d, Y](x)$ span $\mathbb{R}^n$, where we have set $Y=\sum_{i=1}^d \lambda_i V_i$. 

\smallskip

\noindent
This provides a handy geometric interpretation of this assumption and the usual diffusion case tends to indicate that Hypothesis \ref{hyp:ben-arous-leandre} should be minimal in order to establish strict positivity of the density for $X_t^x$.

\smallskip

The second assumption we shall invoke is of elliptic type, and can be stated as follows:
\begin{hypothesis} \label{hyp:elliptic}
The vector fields $V_1,\ldots ,V_d$ of equation \eqref{eq:sde-intro} form an elliptic system, that is,
\begin{equation}\label{eq:hyp-elliptic}
 v^{*} V(x) V^{*}(x) v  \geq \lambda \vert v \vert^2, \qquad \text{for all } v,x \in \R^n,
\end{equation}
where we have set $V=(V_j^i)_{i=1,\ldots ,n; j=1,\ldots d}$ and where $\lambda$ designates  a strictly positive constant. 
\end{hypothesis}

With this set of hypotheses in hand, we obtain the following results:

\smallskip

\noindent
\textbf{(1)}
We first give some general conditions in order to check that the density $p_t$ of $X_t^x$ is strictly positive on $\R^n$:
\begin{theorem}\label{thm:strict-positivity-intro}
Consider the solution $X^{x}$ to equation \eqref{eq:sde-intro} driven by a $d$-dimensional fBm with Hurst parameter $H>1/4$.
Assume that Hypotheses \ref{hyp:regularity-V} and \ref{hyp:ben-arous-leandre} are satisfied, let $t\in(0,1]$ and consider the density $p_t:\R^{n}\to\R_{+}$ of the random variable $X_t^x$. Then  $p_t(y)>0$ for all $y\in\R^n$.
\end{theorem}

\smallskip

\noindent
\textbf{(2)} Next we derive some Gaussian or sub-Gaussian  type upper bounds for the density $p_t$ of the random variable $X^x_t$:
\begin{theorem}\label{thm:upper-bnd-density}
Let $X^x$ be the solution to equation \eqref{eq:sde-intro} driven by a $d$-dimensional fBm $B$ with Hurst parameter $H>1/4$, assume that  $V_1,\ldots ,V_d$ satisfy the  elliptic condition \eqref{eq:hyp-elliptic} and let $t\in(0,1]$. Then the density $p_t$ of $X_t^x$ satisfies the following inequality:
\begin{equation}\label{eq:exp-bound-irregular}
p_t(y) \leq c_{1} t^{-nH} \exp \left(-\frac{\vert y-x \vert^{(2H+1)\wedge 2}}{c_{2} t^{2H} }  \right), \; \;  \text{for all } y \in \R^n,
\end{equation}
for two strictly positive constants $c_1,c_2$.
\end{theorem}
Observe that we have put an emphasis in computing the correct exponents in all terms of relation  \eqref{eq:exp-bound-irregular}. Namely, the terms $t^{-nH}$ and $t^{2H}$ (respectively outside and inside the exponential terms) can be considered as optimal, since they correspond to what one obtains in the fractional Brownian case, i.e non-degenerate constant coefficients $V_1,\ldots ,V_d$ and $V_0\equiv 0$. As far as the exponent of $\vert y-x \vert$ within the exponential is concerned, the quadratic Gaussian term we get in the regular case (namely $H>1/2$) is also optimal, while the exponent $2H+1$ of the irregular case ($H<1/2$) is due to the poorer concentration properties obtained for the Jacobian of equation \eqref{eq:sde-intro}.

\smallskip

\noindent
\textbf{(3)} 
Finally, we complete this paper by studying the relationship between capacities of sets in $\R^{n}$ and hitting probabilities for equation \eqref{eq:sde-intro} seen as a system. Indeed, we are interested in solving a classical problem on potential theory for stochastic processes which is the following: can we relate the hitting probabilities of $X^x$ solution
to  equation \eqref{eq:sde-intro} with a Newtonian capacity? In other words, we wish to know if 
there exists $\alpha \in \R$ such that for all Borel sets $A\subset \R^n$ 
\begin{equation*}
\mathbb{P}(X^x(\R_+) \cap A \neq \varnothing)>0 \quad \Longleftrightarrow  \quad \text{Cap}_\alpha(A) >0 \ .
\end{equation*}

For the sake of readability, let us briefly recall the definition of Newtonian capacity:
for all Borel sets $A\subset \R^n$, we define $\mathcal{P}(A)$ to be the set of all probability measures
with compact support in $A$.
For 
$\mu\in\mathcal{P}(A)$, we let $\mathcal{E}_\alpha(\mu)$ denote the
$\alpha$-dimensional energy of $\mu$, that is,
\begin{equation}\label{eq:def-E-beta-mu}
   \mathcal{E}_\alpha(\mu)  := \iint {\rm K}_\alpha(\vert x-y\vert)\, \mu(dx)\,\mu(dy),
\end{equation}
where ${\rm K}_{\alpha}$ denotes the
$\alpha$-dimensional Newtonian kernel, that is, 
\begin{equation} \label{kernel}
	{\rm K}_\alpha(r) := 
	\begin{cases}
		r^{-\alpha}&\text{if $\alpha >0$},\\
		\log  ( N_0/r ) &\text{if $\alpha =0$},\\
		1&\text{if $\alpha<0$},
    \end{cases}
\end{equation}
where $N_0>0$ is a constant. For all $\alpha\in\R$ and Borel sets
$A\subset\R^n$, we then define the $\alpha$-dimensional capacity of $A$ as
\begin{equation}\label{eq:def-capacity}
    \text{Cap}_\alpha(A) := \left[ \inf_{\mu\in\mathcal{P}(A)}
   \mathcal{E}_\alpha(\mu) \right]^{-1},
\end{equation}
where by convention we set $1/\infty:=0$. In particular, it is easily seen from definitions~\eqref{eq:def-E-beta-mu}--\eqref{eq:def-capacity} that   for any $x\in\R^{n}$ we have $\capa_{\alpha}(\{x\})>0$ if and only if $\alpha<0$.

\smallskip

Let us now go back to our fBm situation: recall that for a
$n$-dimensional fractional Brownian motion $B=(B(t),t\geq0)$ with Hurst
parameter $H\in(0,1)$, the following is well-known (see e.g. \cite{Xiao} and the references therein):
\begin{equation}\label{eq:B-hits-points}
B\text{ hits points in }\mathbb{R}^{n}\text{ a.s.}\text{ if and only if
}n<\frac{1}{H}.
\end{equation}
Moreover, for all $0<a<b$, $\eta>0$, and any Borel set
$A\subset\mathbb{R}^{n}$, there exist constants $c_{3},c_{4}>0$ such that
\begin{equation*}
c_{3} \, \text{Cap}_{n-\frac{1}{H}}(A) \leq\mathbb{P}(B([a,b])\cap
A\neq\varnothing)\leq c_{4}\, \text{Cap}_{n-\frac{1}{H}-\eta}(A). 
\end{equation*}
As in the case of density functions, our aim is to obtain similar  bounds for the solution to equation \eqref{eq:sde-intro}, where $B$ is a fBm with $H>\frac14$. We shall get the following:

\begin{theorem}\label{thm:hitting-capacity-X}
Let $X^x$ be the solution to equation \eqref{eq:sde-intro} driven by a $d$-dimensional fBm $B$ with Hurst parameter $H>1/4$, and let $t\in(0,1]$. Fix $0<a<b\le 1$, $M >0$, and $\eta>0$
 Then whenever $V_1,\ldots ,V_d$ satisfy the elliptic condition \eqref{eq:hyp-elliptic}, there exists two strictly positive constants $c_5,c_6$ depending on $a,b,H,M,n, \eta$ such that for all compact sets $A\subseteq[-M\,,M]^n$,
            \begin{equation}\label{eq:bnd-hitting-regular}
       c_5\,  \textnormal{Cap}_{n-\frac{1}{H} }(A)  \le 
       \mathbb{P}\lp X^x([a,b]) \cap A  \neq\varnothing \rp
       \leq c_6\,  \textnormal{Cap}_{ n-\frac{1}{H}-\eta}(A).
            \end{equation}
\end{theorem}

Moreover, as a corollary of Theorem \ref{thm:hitting-capacity-X}, we easily get that if Hypothesis \ref{hyp:elliptic} is met, then if $n<\frac{1}{H}$
the process $X^x$ hits points in $\mathbb{R}^n$ with strictly positive probability, while if $n>\frac{1}{H}$ the process $X^x$ does
not hit points in $\mathbb{R}^n$ a.s.

\smallskip

Let us say a few words about the methodology we have followed in order to obtain the results above. Our computations lie into the landmark of stochastic analysis for Gaussian processes, and we try to apply general Malliavin calculus tools which yield global recipes in order to get strict positivity  \cite{Nu-flour} or upper bounds \cite[Chapter 2]{Nu06} for densities of random variables defined on the Wiener space. We also invoke the references \cite{DKN,DKN2}, which establish nice relationships between stochastic analysis and potential theory for processes. This being said, our technical efforts will mainly be focused on the following points:
\begin{itemize}
\item 
An accurate Karhunen-Loeve expansion of fBm which will enable us to obtain the strict positivity of the density $p_t$.
\item
A combination of rough paths estimates and a sharp analysis of some covariance matrices related to fBm in order to obtain our exponential upper bounds.
\item
A thorough analysis of bivariate densities for the hitting probabilities of $X^x$.
\end{itemize}
All those points will obviously be detailed in the next sections.

\smallskip

Here is how our article is structured: Section \ref{sec:preliminary-material} gathers some material on fBm and rough differential equations which prove to be useful in the sequel. Section \ref{sec:positivity-density} is devoted to establish criteria for the strict positivity of the density of $X_t^x$ and our Gaussian upper bounds for $p_t$ are handled in Section \ref{sec:upper-bounds}. Finally we get the bounds on hitting probabilities in Section \ref{sec:hitting-probabilities}, where in particular all the previous tools are used.

\smallskip

\noindent
\textbf{Notations:} Throughout this paper, unless otherwise specified, we use $|\cdot|$ for Euclidean norms and $\|\cdot\|_{L^p}$ for the $L^p$ norm with respect to the underlying probability measure $\mp$.

Consider a finite-dimensional vector space $V$. The space of $V$-valued H\"older continuous functions defined on $[0,1]$, with H\"older continuity exponent $\ga\in(0,1)$, will be denoted by $\cac^\ga(V)$, or just $\cac^\ga$ when this does not yield any ambiguity. For a function $g\in\cac^\ga(V)$ and $0\le s<t\le 1$, we shall consider the semi-norms
\begin{equation}\label{eq:def-holder-norms}
\|g\|_{s,t,\ga}=\sup_{s\le u<v\le t}\frac{|g_v-g_u|_{V}}{|v-u|^{\ga}},
\end{equation}
The semi-norm $\|g\|_{0,1,\ga}$ will simply be denoted by $\|g\|_{\ga}$.

Generic universal constants will be denoted by $c,C$ independently of their exact values.

\section{Preliminary material}\label{sec:preliminary-material}

Recall that a fractional Brownian motion $B$ is a $d$-dimensional centered Gaussian process with independent components $B^i$ such that $\EE[(B_t^i-B_s^i)^{2}]$ is given by \eqref{eq:var-increments-B}. Let us also point out that $B$ admits a representation of Volterra type, namely
\begin{equation}\label{eq:volterra-representation}
B_t^i= \iot K(t,u) \, dW_u^i, \quad i=1,\ldots,d,
\end{equation}
for a $d$-dimensional Wiener process $W$ and a kernel $K$ (whose exact expression is given by \eqref{eq:def-kernel-K} below) such that for any $t\in\ou$ we have $K(t,\cdot)\in L^2(\ou)$.   We  denote by $R$ the common covariance of the $B^i$, defined by 
\begin{equation}\label{eq:covariance}
R_{st}= \EE\lc  B_s^i \, B_t^i\rc = \int_0^{s} K(s,u) \, K(t,u) \, du
= \frac12 \lp |t|^{2H} + |s|^{2H} - |t-s|^{2H}\rp,
\end{equation}
for $s,t\in\ou$. In the remainder of the paper we assume that the process $B$ is realized on an abstract Wiener space $(\oom,\cf,\PP)$ with $\oom=\cac_0([0,1];\R^d)$. Namely, $\oom=\cac_0([0,1])$ is the Banach space of continuous functions
vanishing at $0$ equipped with the supremum norm, $\cf$ is the Borel sigma-algebra and $\PP$ is the unique probability measure on $\oom$ such that the canonical process $B=\{B_t=(B^1_t,\ldots,B^d_t), \; t\in [0,1]\}$ is a centered Gaussian process with covariance $R$ given by \eqref{eq:covariance}.

\subsection{Rough path above $\mathbf{B}$}

We consider here $B$ together with its iterated integrals as a rough path, and we refer to \cite{FV-bk,LQ} for further details on this concept. Let us just mention here a few basic facts.

\smallskip

For $N\in\mathbb{N}$, recall that the truncated algebra $T^{N}(\mathbb{R}%
^{d})$ is defined by 
$$
T^{N}(\mathbb{R}^{d})=\bigoplus_{m=0}^{N}(\mathbb{R}%
^{d})^{\otimes m},
$$ 
with the convention $(\mathbb{R}^{d})^{\otimes
0}=\mathbb{R}$. The set $T^{N}(\mathbb{R}^{d})$ is equipped with a straightforward
vector space structure, plus an operation $\otimes$ defined by
\[
\pi_{m}(g\otimes h)=\sum_{k=0}^{N}\pi_{m-k}(g)\otimes\pi_{k}(h),\qquad g,h\in
T^{N}(\mathbb{R}^{d}),
\]
where $\pi_{m}$ designates the projection on the $m$th tensor level. Then
$(T^{N}(\mathbb{R}^{d}),+,\otimes)$ is an associative algebra with unit
element $\mathbf{1} \in(\mathbb{R}^{d})^{\otimes0}$.

\smallskip

For $s<t$ and $m\geq2$, consider the simplex $\Delta_{st}^{m}=\{(u_{1}%
,\ldots,u_{m})\in\lbrack s,t]^{m};\,u_{1}<\cdots<u_{m}\} $, while the
simplices over $[0,1]$ will be denoted by $\Delta^{m}$. A continuous map
$\mathbf{x}:\Delta^{2}\rightarrow T^{N}(\mathbb{R}^{d})$ is called a
multiplicative functional if for $s<u<t$ one has $\mathbf{x}_{s,t}%
=\mathbf{x}_{s,u}\otimes\mathbf{x}_{u,t}$. An important example arises from
considering paths $x$ with finite variation: for $0<s<t$ we set
\begin{equation}
\mathbf{x}_{s,t}^{m}=\sum_{1\leq i_{1},\ldots,i_{m}\leq d}\biggl( \int%
_{\Delta_{st}^{m}}dx^{i_{1}}\cdots dx^{i_{m}}\biggr) \,e_{i_{1}}\otimes
\cdots\otimes e_{i_{m}}, \label{eq:def-iterated-intg}%
\end{equation}
where $\{e_{1},\ldots,e_{d}\}$ denotes the canonical basis of $\mathbb{R}^{d}%
$, and then define the \textit{signature} of $x$ as
\[
S_{N}(x):\Delta^{2}\rightarrow T^{N}(\mathbb{R}^{d}),\qquad(s,t)\mapsto
S_{N}(x)_{s,t}:=1+\sum_{m=1}^{N}\mathbf{x}_{s,t}^{m}.
\]
The function $S_{N}(x)$ for a smooth function $x$ will be our typical example of multiplicative functional. Let us stress the fact that those elements take values in the strict subset
$G^{N}(\mathbb{R}^{d})\subset T^{N}(\mathbb{R}^{d})$ given by the \emph{group-like}
elements
\[
G^{N}(\mathbb{R}^{d}) = \exp^{\oplus}\bigl(L^{N}(\mathbb{R}^{d})\bigr),
\]
where 
\[
L^{N}(\mathbb{R}^{d})=\mathbb{R}^d \oplus [ \mathbb{R}^d, \mathbb{R}^d] \oplus \cdots \oplus [ \mathbb{R}^d, [ \cdots , [ \mathbb{R}^d, \mathbb{R}^d] \cdots ]
\]
and  for two elements in $T^{N}(\mathbb{R}^{d})$, $[a,b]=a \otimes b - b\otimes a$ . This set is called free nilpotent group of step $N$, and is equipped with the classical Carnot-Caratheodory norm which we simply denote by $|\cdot|$. For a path $\mathbf{x}\in\cC([0,1],G^{N}(\R^d))$, the $p$-variation norm of $\mathbf{x}$ is defined to be
\begin{align*}
\|\mathbf{x}\|_{p-{\rm var}; [0,1]}=\sup_{\Pi \subset [0,1]}\left(\sum_i |\mathbf{x}_{t_i}^{-1}\otimes \mathbf{x}_{t_{i+1}}|^p\right)^{1/p}
\end{align*}
where the supremum is taken over all subdivisions $\Pi$ of $[0,1]$.

\smallskip

With these notions in hand, let us briefly define what we mean by geometric rough path (we refer to \cite{FV-bk,LQ} for a complete overview): for $p\ge 1$, an element $x:\ou\to G^{\lfloor p \rfloor}(\R^d)$ is said to be a geometric rough path if it is the $p$-var limit of a sequence $S_{\lfloor p \rfloor}(x^{m})$ of signatures of smooth functions $x^m$. In particular, it is an element of the space 
$$\cC^{p-{\rm var}; [0,1]}([0,1], G^{\lfloor p \rfloor}(\R^d))=\{\mathbf{x}\in \cC([0,1], G^{\lfloor p \rfloor}(\R^d)): \|\mathbf{x}\|_{p-{\rm var}; [0,1]}<\infty\}.
$$

\smallskip

Let us now turn to the fBm case: according to the considerations above, in order to prove that a lift of a $d$-dimensional fBm as a geometric rough path exists it is sufficient to build enough iterated integrals of $B$ by a limiting procedure. Towards this aim, a lot of the information concerning $B$ is encoded in the rectangular increments of the covariance function $R$ (defined by \eqref{eq:covariance}), which are given by
\begin{equation*}
R_{uv}^{st} \equiv \EE\lc (B_t^1-B_s^1) \, (B_v^1-B_u^1) \rc.
\end{equation*}
We then call 2-dimensional $\rho$-variation of $R$ the quantity
\begin{equation*}
V_{\rho}(R)^{\rho} \equiv 
\sup \lcl  
\lp \sum_{i,j} \lln R_{s_{i} s_{i+1}}^{t_{j}t_{j+1}} \rrn^{\rho} \rp^{1/\rho}; \, (s_i), (t_j)\in \Pi
\rcl,
\end{equation*}
where $\Pi$ stands again for the set of partitions of $\ou$. The following result is now well known for fractional Brownian motion:
\begin{proposition}\label{prop:fbm-rough-path}
For a fractional Brownian motion with Hurst parameter $H$, we have $V_{\rho}(R)<\infty$ for all $\rho\ge1/(2H)$. Consequently, for $H>1/4$ the process $B$ admits a lift $\mathbf{B}$ as a geometric rough path of order $p$ for any $p>1/H$.
\end{proposition}

\begin{proof}
The fact that $V_{\rho}(R)<\infty$ for all $\rho\ge1/(2H)$ is the content of \cite[Proposition 15.5]{FV-bk}. The implication on the rough path construction can also be found in \cite[Chapter 15]{FV-bk}.

\end{proof}

\subsection{Malliavin calculus tools}\label{sec:malliavin-tools}
Gaussian techniques are obviously essential in the analysis of densities for solutions to \eqref{eq:sde-intro}, and we proceed here to introduce some of them. These lines follow the classical analysis for Gaussian rough paths as explained in  \cite{FV-bk}.

\subsubsection{Wiener space associated to fBm}\label{sec:wiener-space}
Let $\mathcal{E}$ be the space of $\mathbb{R}^d$-valued step
functions on $[0,1]$, and $\mathcal{H}$  the closure of
$\mathcal{E}$ for the scalar product:
\[
\langle (\mathbf{1}_{[0,t_1]} , \cdots ,
\mathbf{1}_{[0,t_d]}),(\mathbf{1}_{[0,s_1]} , \cdots ,
\mathbf{1}_{[0,s_d]}) \rangle_{\mathcal{H}}=\sum_{i=1}^d
R(t_i,s_i),
\]
where $R$ is defined by \eqref{eq:covariance}.
Then if $(e_1,\ldots,e_d)$ designates the canonical basis of $\R^d$, one constructs an isometry $K^*_H: \ch \rightarrow  L^2([0,1])$  such that $K^*_H(\1_{[0,t]}\, e_{i}) = \1_{[0,t]}$  $K_H(t,\cdot)\,e_{i}$, 
where the kernel $K=K_H$ is given by 
\begin{eqnarray}\label{eq:def-kernel-K}
K(t,s)&=& c_H \, s^{\frac12 -H} \int_s^t (u-s)^{H-\frac 32} u^{H-\frac 12} \, du, \quad H>1/2  \\
K(t,s)&=& c_{H,1} \lp\frac{s}{t} \rp^{1/2-H} \, (t-s)^{H-1/2} 
+ c_{H,2} \, s^{1/2-H} \int_s^t (u-s)^{H-\frac 12} u^{H-\frac 32} \, du, \quad H\le 1/2, \notag
\end{eqnarray}
for some constants $c_H$, $c_{H,1}, c_{H,2}$, and verifies that $\mathbb{E}[B_s^{i}\, B_t^{i}]= \int_0^{s\land t} K(t,r) K(s,r)\, dr$. Moreover, let us observe that the isometry $K^*_H$ alluded to above can be represented 
in the following form by using fractional calculus: for $H>1/2$ we have
\begin{equation}\label{eq:def-K-star}
[K^* \vp]_t = \int_t^1 \vp_r \, \partial_r K(r,t) \, dr
=d_H \, t^{1/2-H} \lc  I_{1^-}^{H-1/2} \lp u^{H-1/2} \vp \rp \rc_t,
\end{equation}
while for $H \le 1/2$ it holds that
\begin{equation}\label{eq:def-K-star-less-half}
[K^* \vp]_t = K(1,t) \, \vp_t 
+ \int_t^1 \lp \vp_r- \vp_t \rp  \, \partial_r K(r,t) \, dr
=d_H \, t^{1/2-H} \lc  D_{1^-}^{1/2-H} \lp u^{H-1/2} \vp \rp \rc_t.
\end{equation}
When $H>\frac{1}{2}$ it can be shown that $L^{1/H} ([0,1])
\subset \mathcal{H}$, and when $H<\frac{1}{2}$ one has 
$C^\gamma\subset \mathcal{H}\subset L^2([0,1])$
for all $\gamma>\frac{1}{2}-H$. We shall also use the following representations of the inner product in $\ch$: For $H>1/2$ and $\phi,\psi\in\ch$, we have 
\begin{equation}\label{eq:inner-pdt-H-smooth}
\langle \phi , \psi \rangle_{\mathcal{H}}=H(2H-1)\int_0^1 \int_0^1
| s-t |^{2H-2} \, \langle \phi_{s} , \psi_{t}\rangle_{\mathbb{R}^d} \, ds dt.
\end{equation}

\smallskip

In order to deduce that $(\oom,\ch,\PP)$ defines an abstract Wiener space, we remark that $\ch$ is continuously and densely embedded in $\oom$. To this aim define first the space $\bch$ as 
\begin{equation*}
\bch = \lcl  
\ell:[0,1]\to\R^d ; \, \ell_{t} = \int_{0}^{t} K(t,u) \, \phi_{u} \, du,
\text{ with } \phi\in L^{2}([0,1])
\rcl,
\end{equation*}
where $K$ is defined by \eqref{eq:def-kernel-K}. It is worth noticing at this point that the space $\bch$ yields the accurate notion of Cameron-Martin space in the fBm context (for Brownian motion one obtains $\ch=L^2(\ou)$ and $\bch=W^{1,2}(\ou)$). Then one proves that the operator $\crr:=\crr_H :\ch \rightarrow \bch$ given by
\begin{equation}\label{eq:def-R}
\crr \psi := \int_0^\cdot K(\cdot,s) [K^*_H \psi](s)\, ds
\end{equation}
defines a dense and continuous embedding from $\ch$ into $\oom$; this is due to the fact that $\crr_H \psi$ is $H$-H\"older continuous (for details, see \cite[p. 399]{NS}).  Let us now quote from \cite[Chapter 15]{FV-bk} a result relating the 2-d regularity of $R$ and the regularity of $\bch$.
\begin{proposition}\label{prop:imbed-bar-H}
Let $B$ be a fBm with Hurst parameter $H\in(1/4,1/2)$. Then one has $\bch\subset \cac^{\rho-{\rm var}}$ for $\rho>(H+1/2)^{-1}$. Furthermore, the following quantitative bound holds:
\begin{equation*}
\Vert h \Vert_{\bch} \ge \frac{\Vert h \Vert_{\rho-{\rm var}}}{(V_\rho(R))^{1/2}}.
\end{equation*}
\end{proposition}
As the reader might have observed, there is a substantial gain in talking about $p$-variations instead of H\"older norms in this context. Indeed, for fBm we have $\bch\subset \cac^{\rho-{\rm var}}$ for $\rho>(H+1/2)^{-1}$ while we only have $\bch\subset \cac^{H}$. This means that functions in $\bch$ are more than twice as regular in terms of $p$-variations than in terms of H\"older norms. Furthermore, an integral of the form $\int h \, dB$ can be interpreted in the Young sense by means of $p$-variation techniques.

\smallskip

Let us close this section by pointing out an implication of Volterra's representation of fBm~\eqref{eq:volterra-representation} in terms of filtrations. Indeed, it is readily checked that $\cf_t\equiv\si(\{B_s; \, 0\le s\le t\})$ can also be expressed as $\cf_t=\si(\{W_s; \, 0\le s\le t\})$. This filtration will be important in the sequel.

\subsubsection{Scale invariant inequalities}

The following  inequalities, in particular the lower bounds, shall be used several times throughout the text. They show that one can replace the $\mathcal{H}$-norm that may be difficult to estimate by simpler quantities while keeping the correct scaling in time.

\begin{proposition}\label{interpolation}
Let $\ch$ be the Hilbert space introduced at Section \ref{sec:wiener-space}, depending on the Hurst parameter $H\in(0,1)$. Then:

\begin{itemize}
\item Asume $H>1/2$. Let $\gamma > H-1/2$. There exist constants $c_1,c_2 >0$ such that for every continuous $f \in \mathcal{H}$, and $t \in (0,1]$,
\[
c_1 t^{2H} \frac{\min_{[0,1]}  |f|^4}{ \|f\|^2_{\infty}+\| f \|_\gamma^2} \le \| f \mathbf{1}_{[0,t]} \|_\mathcal{H}^2 \le c_2 t^{2H} \| f \|^2_\infty.
\]
\item Assume $H \le 1/2$ and let $\gamma > 1/2-H$. There exist  constants $c_1,c_2 >0$ such that for every $f \in C^\gamma$,  and $t \in (0,1]$,
\[
c_1 t^{2H}\min_{[0,1]} |f|^2  \le  \| f \mathbf{1}_{[0,t]}  \|_\mathcal{H}^2\le c_2 t^{2H} (\| f\|_\gamma^2+\|f \|_\infty^2) .
\]
\end{itemize}
\end{proposition}

\begin{proof}
We first assume $H>1/2$. The inequality $ \| f \mathbf{1}_{[0,t]} \|_\mathcal{H}^2 \le c_2 t^{2H} \| f \|^2_\infty$ is a straightforward consequence of $\eqref{eq:inner-pdt-H-smooth}$. The inequality
\[
c_1 \frac{\min_{[0,1]}  |f|^4}{ \|f\|^2_{\infty}+\| f \|_\gamma^2} \le \| f \|_\mathcal{H}^2
\]
is proved in  \cite[Lemma 4.4]{BH}. For $t\in(0,1]$, this inequality can be rescaled as follows,
\begin{align*}
 \| f \mathbf{1}_{[0,t]} \|_\mathcal{H}^2 &= H(2H-1)\int_0^t \int_0^t
| u-v |^{2H-2} \langle f(u), f(v)\rangle du dv \\
 &= H(2H-1)  t^{2H}\int_0^1 \int_0^1
| u-v |^{2H-2} \langle f(tu), f(tv)\rangle du dv  \\
& \ge c_1 t^{2H} \frac{\min_{[0,1]}  |f_t|^4}{ \|f_t\|^2_{\infty}+\| f_t \|_\gamma^2}\\
& \ge  c_1 t^{2H} \frac{\min_{[0,1]}  |f|^4}{ \|f\|^2_{\infty}+\| f \|_\gamma^2},
\end{align*}
where $f_t(u)=f(tu)$. This proves our claim for $H>1/2$.

\smallskip

We now assume $H \le 1/2$. The fact that $  \| f \|_\mathcal{H}^2 \ge c_1 \| f \|_2^2 \ge c_1 \min_{[0,1]} |f|^2 $ is well known and the inequality easily rescales as above. The last inequality to prove is the upper bound. It is pointed in  \cite{NS} that we have, for any $h_1,h_2 \in\ch$,
\[
\langle h_1, h_2 \rangle_\ch =\int_0^1 h_1 d \crr h_2,
\]
where the right hand side is understood in the Young sense and $\crr$ is the isometry going from $\ch$ to $\bch$. Hence, if $p^{-1}+q^{-1}>1$ and $p>H^{-1}$, $q> (1/2+H)^{-1}$ we have
\begin{align*}
| \langle h_1, h_2 \rangle_\ch | &\le C( \| h_1\|_{p-var} +\| h_1\|_{\infty} ) \| \crr h_2 \|_{q-var}.
\end{align*}
We now use Proposition \ref{prop:imbed-bar-H} to get the bound
\[
 \| \crr h_2 \|_{q-var} \le C  \| \crr h_2 \|_{\bar{\ch}} =C \| h_2 \|_\ch.
\]
This proves that
\[
 \| f   \|_\mathcal{H}^2\le c_2  (\| f\|_\gamma^2+\|f \|_\infty^2)
\]
Again, this inequality easily rescales on the time interval $[0,t]$.
\end{proof}

\subsubsection{Malliavin calculus for $B$}
A $\mathcal{F}_{1}$-measurable real
valued random variable $F$ is said to be cylindrical if it can be
written, for a given $m\ge 1$, as
\begin{equation*}
F=f\lp  B_{t_1},\ldots,B_{t_m}\rp,
\quad\mbox{for}\quad
0\le t_1<\cdots<t_m \le 1,
\end{equation*}
where $f:\mathbb{R}^m \rightarrow \mathbb{R}$ is a $C_b^{\infty}$ function. The set of cylindrical random variables is denoted by $\mathcal{S}$. 

\smallskip

The Malliavin derivative is defined as follows: for $F \in \mathcal{S}$, the derivative of $F$ in the direction $h\in\ch$ is given by
\[
\mathbf{D}_h F=\sum_{i=1}^{m}  \frac{\partial f}{\partial
x_i} \left( B_{t_1},\ldots,B_{t_m}  \right) \, h_{t_i}.
\]
More generally, we can introduce iterated derivatives. If $F \in
\mathcal{S}$, we set
\[
\mathbf{D}^k_{h_1,\ldots,h_k} F = \mathbf{D}_{h_1} \ldots\mathbf{D}_{h_k} F.
\]
For any $p \geq 1$, it can be checked that the operator $\mathbf{D}^k$ is closable from
$\mathcal{S}$ into $\mathbf{L}^p(\oom;\ch^{\otimes k})$. We denote by
$\mathbb{D}^{k,p}(\ch)$ the closure of the class of
cylindrical random variables with respect to the norm
\[
\left\| F\right\| _{k,p}=\left( \mathbb{E}\left[|F|^{p}\right]
+\sum_{j=1}^k \mathbb{E}\left[ \left\| \mathbf{D}^j F\right\|
_{\ch^{\otimes j}}^{p}\right] \right) ^{\frac{1}{p}},
\]
and $\mathbb{D}^{\infty}(\ch)=\cap_{p \geq 1} \cap_{k\geq 1} \mathbb{D}^{k,p}(\ch)$.

\smallskip

Estimates of Malliavin derivatives are crucial in order to get information about densities of random variables, and Malliavin matrices as well as non-degenerate random variables will feature importantly in the sequel:
\begin{definition}\label{non-deg}
Let $F=(F^1,\ldots , F^n)$ be a random vector whose components are in $\mathbb{D}^\infty(\ch)$. Define the Malliavin matrix of $F$ by
\begin{equation} \label{malmat}
\gamma_F=(\langle \mathbf{D}F^i, \mathbf{D}F^j\rangle_{\ch})_{1\leq i,j\leq n}.
\end{equation}
Then $F$ is called  {\it non-degenerate} if $\gamma_F$ is invertible $a.s.$ and
$$(\det \gamma_F)^{-1}\in \cap_{p\geq1}L^p(\Omega).$$
\end{definition}
\noindent
It is a classical result that the law of a non-degenerate random vector $F=(F^1, \ldots , F^n)$ admits a smooth density with respect to the Lebesgue measure on $\mr^n$. Furthermore, the following integration by parts formula allows to get more quantitative estimates:
\begin{proposition} \cite[Proposition 2.1.4]{Nu06} \label{intbyparts}
Let $F=(F^1,\ldots ,F^n)$ be a non-degenerate random vector as in Definition \ref{non-deg}. 
Let $G \in \md^{\infty}$ and $\varphi$ be a function in the space $C_p^\infty(\R^n)$. Then for any multi-index $\alpha \in \{1,2,\ldots ,n\}^k$, $k \geq 1$, there exists
an element $H_{\alpha}(F,G) \in \md^{\infty}$ such that
\begin{equation*}
\me[\partial_{\alpha} \varphi(F) G] =\me[\varphi(F) H_{\alpha}(F,G)],
\end{equation*}
Moreover, the elements $H_{\alpha}(F,G)$ are recursively given by
\begin{equation}\label{eq:recursive-H-alpha}
H_{(i)}(F,G) =\sum_{j=1}^n \delta\left( G (\gamma^{-1}_F)^{ij} \, \bd F^j \right) 
\quad\text{and}\quad
H_{\alpha}(F,G)= H_{\alpha_k}(F,H_{(\alpha_1,\ldots ,\alpha_{k-1})}(F,G)),
\end{equation}
and for $1 \leq p<q<\infty$ we have
\begin{equation} \label{Holder}
\Vert  H_{\alpha}(F,G) \Vert_p \leq c_{p,q} \Vert \gamma^{-1}_F \, \bd F\Vert^k_{k, 2^{k-1}r}\Vert G\Vert^k_{k, q},
\end{equation}
 where $\frac1p=\frac1q+\frac1r$.
\end{proposition}

As a consequence, one has the following expression for the density of a non-degenerate random vector.
\begin{proposition} \label{density} \cite[Proposition 2.1.5]{Nu06}
Let $F=(F^1,\ldots ,F^n)$ be a non-degenerate random vector as in Definition \ref{non-deg}. Then the density $p_F(y)$ of $F$ belongs to the Schwartz space, and
for any $\sigma \subset \{1,\ldots ,n\}$,
\begin{equation*}
p_F(y)=(-1)^{n-\vert \sigma \vert}\me[{\bf 1}_{\{ F^i>y^i, i \in \sigma, F^i <y^i, i \neq \sigma\}} H_{(1,\ldots ,n)}(F,1)], \; \;  \text{for all } y \in \R^n. 
\end{equation*}
\end{proposition}

\subsubsection{Karhunen-Loeve expansions}
Karhunen-Loeve expansions are approximations of the Gaussian process $B$ in $\bch$. We shall design here one of those expansions, which will be useful for further computations. It relies on the Volterra type representation \eqref{eq:volterra-representation} for $B$.

\smallskip

To this aim, consider the Cameron-Martin space $\bch^{W}$ of the usual Brownian motion, namely $\bch^{W}=W^{1,2}([0,1])$, and let $(h_k)_{k \ge 1}$ be any orthonormal  basis of $\bch^{W}$. If $\{Z_k; \, k\ge 1\}$ is an i.i.d sequence of standard Gaussian random variables, it is well-known (see e.g~\cite{Us-bk}) that the process
\[
W_t=\sum_{k=1}^{+\infty} h_k(t) Z_k
\]
is a Brownian motion on $[0,1]$. Our Karhunen-Loeve approximation of $W$ will be given by $W^n_t=\sum_{k=1}^{n} h_k(t) Z_k$, and we have the following result:

\begin{proposition}\label{prop:good-karhunen}
Let $0 < \tau <1$. There exists an orthonormal basis $\{\ell_k; \, k\ge 1\}$ of $\bch^{W}$ such that, setting $W^n_t=\sum_{k=1}^{n} \ell_k(t) Z_k$, the distribution of the processes $W$ and $W-W^n$ are equivalent on $[0,\tau]$.
\end{proposition}

\begin{proof}
We divide this proof in two steps.

\smallskip

\noindent
\textit{Step 1:} We first prove that if the matrix $(\int_\tau^1 \ell_i'(s)\ell_j'(s) ds )_{1 \le i,j \le n}$ is invertible, then the distribution of the processes $W$ and $W-W^n$ are equivalent on $[0,\tau]$.

\smallskip

For this, let us first observe that $W-W^n$ has the same  distribution as the Brownian motion $W$ conditioned by the event $(\int_0^1 \ell'_k(s) dW_s=0, \text{ for } 1\le k \le n)$. Indeed, for any bounded and measurable functional $F$ on the Wiener space, we have
\begin{align*}
 & \mathbb{E} \left[ F\left( W_t , 0 \le t \le 1\right) \Big\vert\, \int_0^1 \ell'_k(s) dW_s=0, \, 1\le k \le n \right] \\
 =&\mathbb{E} \left[ F\left( \sum_{k=1}^{+\infty} \ell_k(t) Z_k , 0 \le t \le 1\right) \Big\vert\, \int_0^1 \ell'_k(s) dW_s=0, \, 1\le k \le n \right] \\
 =&\mathbb{E} \left[ F\left( \sum_{k=n+1}^{+\infty} \ell_k(t) Z_k , 0 \le t \le 1\right) \Big\vert\, \int_0^1 \ell'_k(s) dW_s=0, \, 1\le k \le n \right] \\
 =&\mathbb{E} \left[ F\left( \sum_{k=n+1}^{+\infty} \ell_k(t) Z_k , 0 \le t \le 1\right) \right], 
\end{align*}
where we have invoked the independence of the families $\{\int_0^1 \ell'_k(s) dW_s;\, 1\le k \le n\}$ and $\{\int_0^1 \ell'_k(s) dW_s;\,  k > n\}$. It is thus readily checked that
\begin{equation} \label{eq:cdt-expectation1}
\mathbb{E} \left[ F\left( W_t , 0 \le t \le 1\right) \Big\vert\, \int_0^1 \ell'_k(s) dW_s=0, \, 1\le k \le n \right]  
= \mathbb{E} \left[ F\left( W_t-W^n_t , 0 \le t \le 1\right)\right].
\end{equation}

Let now $0 < \tau <1$ and assume that the matrix $(\int_\tau^1 \ell_i'(s)\ell_j'(s) ds )_{1 \le i,j \le n}$ is invertible. This invertibility implies that the conditional density of $ ( \int_0^1 \ell'_k(s) dW_s)_{ 1\le k \le n}$ given $\sigma(W_s, s \le \tau)$ with respect to the distribution of  $ ( \int_0^1 \ell'_k(s) dW_s)_{ 1\le k \le n}$   exists. Let us denote by $\eta_\tau (y)$, $y \in \mathbb{R}^n$ this density. If $F$ is a bounded and measurable functional on the Wiener space we then have
\begin{equation}\label{eq:cdt-expectation2}
\mathbb{E} \left[ F\left( W_t , 0 \le t \le \tau \right) \Big\vert\, \int_0^1 \ell'_k(s) dW_s=0, 1\le k \le n \right] =
\mathbb{E} \left[ \eta_\tau (0)F\left( W_t , 0 \le t \le \tau \right) \right].
\end{equation}
Gathering relations \eqref{eq:cdt-expectation1} and \eqref{eq:cdt-expectation2}, we thus get that the distribution of the processes $W-W^n$ and $W$ are equivalent on $[0,\tau]$. Our proposition is thus proved once we show that there exists an orthonormal basis $\{\ell_k; \, k\ge 1\}$ of $\bch^{W}$ such that for any $\tau\in[0,1)$, the matrix $(\int_\tau^1 \ell_i'(s)\ell_j'(s) ds )_{1 \le i,j \le n}$ is invertible.

\smallskip

\noindent
\textit{Step 2:}
Let us now construct an orthonormal basis of $\bch^{W}$ with the desired invertibility property: let $(f_k)_{k \ge 1}$ be any basis of $L^2[0,1]$ and denote by $\ell'_k$ the Gram-Schmidt orthonormalisation of $(f_k)_{k \ge 1}$ . By using triangular matrices, we see that the invertibility  of the matrix $(\int_\tau^1 \ell_i'(s)\ell_j'(s) ds )_{1 \le i,j \le n}$  is then equivalent to the invertibility of $(\int_\tau^1 f_i(s)f_j(s) ds )_{1 \le i,j \le n}$. For instance, by choosing $f_k(t)=(1-t)^{k-1}$, $k \ge 1 $, some elementary calculations involving Hilbert matrices yield our claim.

\end{proof}

The previous result on Brownian motion has a direct implication in terms of our fractional Brownian motion $B$:
\begin{corollary}\label{cor:def-B-n}
Let $0 < \tau <1$. There exists an orthonormal basis $\{h_k; \, k\ge 1\}$ of $\bch$ such that, setting $B^n_t=\sum_{k=1}^{n} h_k(t) Z_k$, the distribution of the processes $B$ and $B-B^n$ are equivalent on $[0,\tau]$.
\end{corollary}

\begin{proof}
Take the orthonormal basis $\{\ell_k; \, k\ge 1\}$ of $\bch^{W}$ constructed at Proposition \ref{prop:good-karhunen} and set $h_k(t)=\iot K(t,u) \ell'_{k}(u) \, du$.

\end{proof}

\subsection{Differential equations driven by fractional Brownian motion}
Recall that we consider the following kind of equation:
\begin{equation}
\label{eq:sde} X^{x}_t =x +\int_0^t V_0 (X^x_s)ds+
\sum_{i=1}^d \int_0^t V_i (X^{x}_s) dB^i_s,
\end{equation}
where the vector fields $V_0,\ldots,V_d$ are $\cac_b^\infty$-vector fields on $\R^n$ and $B$ is our driving fBm as defined in \eqref{eq:volterra-representation}.

\subsubsection{Existence, uniqueness and estimates}
Proposition \ref{prop:fbm-rough-path} ensures the existence of a lift of $B$ as a geometrical rough path. The general rough paths theory (see e.g.~\cite{FV-bk,Gu}) allows thus to state  the following proposition:

\begin{proposition}\label{prop:moments-sdes-rough}
Consider equation (\ref{eq:sde}) driven by a $d$-dimensional fBm $B$ with Hurst parameter $H>1/4$, and assume that the vector fields $V$ satisfy Hypothesis \ref{hyp:regularity-V}. Then

\smallskip

\noindent
\emph{(i)}
Equation (\ref{eq:sde}) admits a unique finite $p$-var continuous solution $X^x$ in the rough paths sense, for any $p> 1/H$. 

\smallskip

\noindent
\emph{(ii)}
For any $\lambda>0$ and $\delta<1/p$ we have
\begin{equation}\label{eq:exp-delta-moments}
\me\left[\exp\lambda\left(\sup_{0\leq t\leq T}|X^x_t|^\delta\right)\right]<\infty.
\end{equation}
\end{proposition}

In fact inequality \eqref{eq:exp-delta-moments} can be improved to get the following exponential bound:
\begin{proposition}\label{prop:exp-moments-rdes}
Under the assumptions of Proposition \ref{prop:moments-sdes-rough}, the following inequality holds true: 
\begin{equation}\label{eq:concentration-X}
\PP\lp \sup_{t\in[0,1]} |X^x_t-x| \ge \xi \rp \le  \exp\lp -\frac{c_{H} \,  \xi^{(2H+1)\wedge 2}}{t^{2H}} \rp.
\end{equation}
\end{proposition}

\begin{proof}
Consider first the case $1/4<H<1/2$. Taking up the notation of \cite{CLL} we consider $p>2\rho$ and the control
\begin{equation}\label{eq:def-control-B}
\om_{\bb,p}(s,t)=\|\bb\|_{p-{\rm var};[s,t]}^{p}.
\end{equation}
Then \cite[Lemma 10.7]{FV-bk} states that
\begin{equation}\label{eq:bnd-X-pvar}
\|X^x\|_{p-{\rm var};[s,t]} \le c_V \lp \|\bb\|_{p-{\rm var};[s,t]} \vee \|\bb\|_{p-{\rm var};[s,t]}^{p} \rp
= c_V \lp \lc \om_{\bb,p}(s,t)\rc^{1/p} \vee \om_{\bb,p}(s,t) \rp.
\end{equation}
In particular, for any $t_{i}<t_{i+1}$ we have
\begin{equation}\label{eq:ineq-davies-increments}
|\der X^x_{t_{i}t_{i+1}}| \le 
c_V \lp \lc \om_{\bb,p}(t_{i},t_{i+1})\rc^{1/p} \vee \om_{\bb,p}(t_{i},t_{i+1}) \rp .
\end{equation}
Consider now $\alpha\ge 1$ and construct a partition of $\ot$ inductively in the following way: we set $t_0=0$ and
\begin{equation}\label{eq:def-tau-i}
t_{i+1}= \inf\lcl  u >t_{i} ; \, \|\bb\|^p_{p-{\rm var};[t_{i},u]} \ge \alpha \rcl.
\end{equation}
We then set $N_{\alpha,t,p}=\sup\{n\ge 0;\, t_n < t\}$. Observe that, since we have taken $\alpha\ge 1$, inequality~\eqref{eq:ineq-davies-increments} can be read as $|\der X_{t_{i}t_{i+1}}| \le  c_V \,\om_{\bb,p}(t_{i},t_{i+1}) = c_V \,\alpha$. Hence
\begin{equation}\label{eq:telescopic-increments-X}
|X^x_t-x| \le | X_t^x-X_{N_{\alpha,t,p} } |+ \sum_{i=0}^{N_{\alpha,t,p}-1} |\der X_{t_{i}t_{i+1}}| \le c_V \,\alpha \, (N_{\alpha,t,p}+1).
\end{equation}
Recall now Theorem 6.4 in \cite{CLL}: we have
\begin{equation}\label{eq:concentration-N}
\PP \lp N_{\alpha,t,p} +1> n \rp
\lesssim \exp\lp  -\frac{c_{p,\rho} \, n^{2/\rho}}{t^{2H}}\rp,
\end{equation}
where $\rho=(H+1/2)^{-1}$ is the constant introduced at Proposition \ref{prop:imbed-bar-H}. This easily yields
\begin{equation}\label{eq:concentration-X-2}
\PP\lp \sup_{t\in[0,1]} |X^x_t-x| \ge \xi \rp
\le \PP \lp c_V \,\alpha \, (N_{\alpha,t,p} +1)> \xi \rp
\lesssim \exp\lp -\frac{c_{p,\rho,V} \, \alpha^{2-2/\rho} \xi^{2/\rho}}{t^{2H}} \rp,
\end{equation}
which is our claim. The case $H>1/2$ is handled along the same lines, except that the coefficient $n^{2/\rho}$ in \eqref{eq:concentration-N} is replaced by $n^2$, which reflects into the fact that $\xi^{2/\rho}$ in \eqref{eq:concentration-X-2} is replaced by $\xi^{2}$.
\end{proof}

\subsubsection{Differentiability}

Once equation (\ref{eq:sde}) is solved, the vector $X_t^x$ is a typical example of  random variable which can be differentiated in the Malliavin sense. We shall express this Malliavin derivative in terms of the Jacobian $\bj$ of the equation, which is defined by the relation $\bj_{t}^{ij}=\partial_{x_j}X_t^{x,i}$. Setting $DV_{j}$ for the Jacobian of $V_{j}$ seen as a function from $\R^{n}$ to $\R^{n}$, let us recall that $\bj$ is the unique solution to the linear equation
\begin{equation}\label{eq:jacobian}
\bj_{t} = \id_{n} + \int_0^t DV_0 (X^x_s) \, \bj_{s} \, ds+
\sum_{j=1}^d \int_0^t V_j (X^{x}_s) \, \bj_{s} \, dB^j_s,
\end{equation}
and that the following results hold true (see \cite{CF} and \cite{NS}  for further details):
\begin{proposition}\label{prop:deriv-sde}
Let $X^x$ be the solution to equation (\ref{eq:sde}) and suppose the $V_i$'s satisfy Hypothesis \ref{hyp:regularity-V}. Then
for every $i=1,\ldots,n$, $t>0$, and $x \in \mathbb{R}^n$, we have $X_t^{x,i} \in
\mathbb{D}^{\infty}(\ch)$ and
\begin{equation*}
\mathbf{D}^j_s X_t^{x}= \mathbf{J}_{s,t} V_j (X^x_s) , \quad j=1,\ldots,d, \quad 
0\leq s \leq t,
\end{equation*}
where $\mathbf{D}^j_s X^{x,i}_t $ is the $j$-th component of
$\mathbf{D}_s X^{x,i}_t$, $\mathbf{J}_{t}=\partial_{x} X^x_t$ and $\bj_{s,t}=\bj_{t}\bj_{s}^{-1}$. 
%\txx{blue}{Here I've changed a little the notation, with $\bj_{s,t}$ instead of $\bj_{st}$. This is due to later computations, where $\bj_{st}$ meant $\bj$ evaluated at time $s\times t$, which was confusing.}
\end{proposition}

\smallskip

Let us now quote the recent result \cite{CLL}, which gives a useful estimate for moments of the Jacobian of rough differential equations driven by Gaussian processes.
\begin{proposition}\label{prop:moments-jacobian}
Consider a  fractional Brownian motion $B$ with Hurst parameter $H\in (1/4,1/2]$ and $p>1/H$. Then for any $\eta\ge 1$, there exists a finite constant $c_\eta$ such that the Jacobian $\bj$ defined at Proposition \ref{prop:deriv-sde} satisfies:
\begin{equation}\label{eq:moments-J-pvar}
\EE\lc  \Vert \bj \Vert^{\eta}_{p-{\rm var}; [0,1]} \rc = c_\eta.
\end{equation}
\end{proposition}

%We also record the basic estimate which is an easy consequence of Proposition \ref{interpolation}..

%\begin{proposition}\label{rmk:bnd-norm-H-pvar}
%For any $\eta\geq1$, there exists a constant $c_\eta>0$ such that for $t \in [0,1]$,
%\begin{align*}
%\EE \lc \Vert \bd X^x_t \Vert_{\ch}^{\eta}\rc \leq c_{\eta} \, t^{\eta H}.
%\end{align*}
%\end{proposition}

\section{Strict positivity of the density}
\label{sec:positivity-density} 
In this section, we follow the approach developed by Ben Arous and L\'{e}andre \cite{BL} and prove the strict positivity of the density of solutions to equation (\ref{eq:sde}) as stated in Theorem~\ref{thm:strict-positivity-intro}. 
We first present,  at Section \ref{sec:strict-positivity-general}, the general criterion characterizing the set of points where the density is strictly positive for a non-degenerate finite-dimensional random variable $F$. Then we show how to apply this criterion in our fractional SDE context at Section \ref{sec:positive-frac-sdes}.

\subsection{Strict positivity of the density for non-degenerate random variables}
\label{sec:strict-positivity-general}

We borrow the considerations here from \cite{Nu-flour}, for which we refer for further details. Consider $(\oom,\cf,\PP)$ the canonical probability space associated with our fBm $B$.

\smallskip

Let us now introduce, for a given element $\underline{\ell}=(\ell_1,\ldots,\ell_n)\in\ch^n$ and a vector $z\in\mr^n$, the shifted Gaussian process
$$(T_z^{\underline{\ell}}B)(h)=B(h)+\sum_{j=1}^nz_j\langle h,\ell_{j}\rangle_{\ch} ,\quad h\in\ch .$$
Cameron-Martin's theorem of change of measures shows that for any integrable random variable $G$ we have $\me [G]=\me [G(T_z^{\underline{\ell}}B)J_z]$, where
$$J_z=\exp\left(-\sum_{j=1}^n z_jB(\ell_{j})-\frac{1}{2}\left\|\sum_{j=1}^n z_j\ell_{j}\right\|^2_{\ch }\right).$$
With the same $\underline{\ell}=(\ell_1,\ldots,\ell_n)$ as above, for any multi-index $\alpha=(\alpha_1,\ldots,\alpha_k)$ lying in $\{1,2,\ldots,n\}^k$, let $\underline{\ell}_\alpha=(\ell_{\alpha _1},\ldots,\ell_{\alpha _k})$ and define
$$R_{\underline{\ell}_\alpha,p}F=\int_{\{|z|\leq1\}}\left\langle(\mathbf{D}^kF)(T_z^{\underline{\ell}}B), \ell_{\alpha _1}\otimes\cdot\cdot\cdot\otimes \ell_{\alpha _k}\right\rangle^p_{\ch ^{\otimes k}}dz,$$
for some $p>n$ and multi-index $\alpha$ with $|\alpha|=k\geq 0$.

\smallskip

With these notations in mind, our general criterion for positivity of densities can be read as follows:
\begin{theorem}\label{th: main-sufficient}
Let $F=(F^1,\ldots,F^n)$ be a non-degenerate random variable and $\Phi:\ch \to\mr^n$ a $\cac^{\infty}$ functional. Suppose that the following condition holds:

\smallskip

\noindent
\emph{\bf (H1)}
For any $h\in\ch $ there exists a sequence of measurable transformations $T^h_N: \Omega\to\Omega$ such that $\mp\circ(T^h_N)^{-1}$ is absolutely continuous with respect to $\mp$. Moreover, let $\{\mathbf{D}\Phi^{j}(h) ; \, j=1,\ldots, n\}$ be the coordinates of $\Phi(h)$ in $\R^{n}$, and set $\underline{\ell}=(\mathbf{D}\Phi^1(h),\ldots,\mathbf{D}\Phi^n(h))$. Then for every $\varepsilon>0$ we suppose that we have
\begin{enumerate}
\item $\lim_{N\to\infty}\mp\{|F\circ T^h_N-\Phi(h)|>\varepsilon\}=0$;
\item $\lim_{N\to\infty}\mp\{\|(\mathbf{D}F)\circ T^h_N- (\mathbf{D}\Phi)(h)\|_{\ch }>\varepsilon\}=0$; and
\item $\lim_{M\to\infty}\sup_{N}\mp\{(R_{\underline{\ell}_\alpha,p}F)\circ T^h_N>M\}=0$ for some $p>n$ and all multi-index $\alpha$ with $|\alpha|=0,1,2,3.$
\end{enumerate}
Finally, for a fixed $y\in\mr^n$ assume that there exists  an $h\in\ch $ such that $\Phi(h)=y$ and for the deterministic Malliavin matrix $\gamma_{\Phi}(h)$ of $\Phi$ at $h$, one has $\det \gamma_{\Phi}(h)>0$. Then the density of $F$ satisfies $p(y)>0$.
\end{theorem}

\begin{proof}
The theorem is borrowed from \cite{Nu-flour}, with a slight modification of the definition of  $R_{\underline{\ell}_\alpha,p}F$. The legitimacy of making such modification is seen directly from the proof of Proposition 4.2.2 in \cite{Nu-flour}. 
\end{proof}

\subsection{Strict positivity of the density for solutions to fractional SDE's}
\label{sec:positive-frac-sdes}
This section is devoted to the proof of Theorem \ref{thm:strict-positivity-intro}. The idea is to apply the general Theorem \ref{th: main-sufficient} to $F=X_t^x$ for each fixed $t>0$, where $X^x$ is the solution to equation \eqref{eq:sde} and where we still work under Hypotheses \ref{hyp:regularity-V} and \ref{hyp:ben-arous-leandre}. 
In this context, some natural definitions of the maps $T^h_N$ and of the functional $\Phi$ are as follows: 

\smallskip

\noindent
\textbf{(i)} For any $h\in\ch$, we simply define $T^h_N$ by the identity
$$ 
T^h_N(B)=B-B^N+ \crr h,
$$
where $B^N$ has been defined at Proposition \ref{prop:good-karhunen} and Corollary \ref{cor:def-B-n} and with $\crr h^i$ defined by~\eqref{eq:def-R}.

\smallskip

\noindent
\textbf{(ii)}
The map $\Phi$ is defined as the evaluation of a function at $t\in(0,1]$. Namely, $\Phi(h)$ is solution to the ordinary differential equation
\begin{equation}\label{eq:skeleton}
\Phi(h)_t=x+\int_0^tV_0(\Phi(h)_s)ds+\sum_{i=1}^d\int_0^tV_i(\Phi(h)_s)d \crr h^i_s,
\end{equation}
understood in the {($p$-var)} Young sense.

\smallskip

\noindent
In  what follows, we need to check the above $\Phi$ and $T^h_N$ satisfy condition $\bf(H1)$ in Theorem~\ref{th: main-sufficient}.

\smallskip

Recall that, according to Proposition \ref{prop:fbm-rough-path},  $B$ admits a lift to $G^{\lfloor p \rfloor}(\R^d)$  as a geometric rough path for any fixed $p>1/H$. If $B^N$ is the Karhunen-Loeve type approximation of $B$ discussed above, denote by $\tilde{\mathbf{B}}^N$ the lift of $\tilde{B}^N=B-B^N$ to $G^{\lfloor p \rfloor}(\R^d)$. We have
\begin{proposition}\label{th: rough-approx1}
There exists constant $\eta>0$ depending on $p, \rho$ and the process $B$ such that
$$\sup_N\me \lc \exp\left(\eta\|\tilde{\mathbf{B}}^N\|^2_{p-{\rm var}; [0,1]}\right) \rc <\infty.$$
Moreover, for all $q\geq 1$, 
$$\|\tilde{\mathbf{B}}^N\|_{p-{\rm var}; [0,1]}\to 0\quad\mathrm{in}\ L^q(\mp)\ \mathrm{as}\ N\to\infty.$$ 
\end{proposition}
\begin{proof}
The Gaussian tail of $\|\tilde{\mathbf{B}}^N\|_{p-{\rm var}; [0,1]}$ follows from Lemma 15.46 as well as Proposition~15.22 in \cite{FV-bk}. The rest of the statement is the content of Theorem 15.47 in \cite{FV-bk}.

\end{proof}

We also need the following lemma which is a restatement of Theorem 9.33 and Corollary~9.35 in~\cite{FV-bk}.
\begin{lemma}\label{th: rough-couple-conti}
For any $1\leq q\leq p$ so that $p^{-1}+ q^{-1}>1$, let $(\mathbf{x}, h)\in \cC^{p-{\rm var}}([0,1], G^{\lfloor p \rfloor}(\R^d))\times \cC^{q-{\rm var}}([0,1], \R^d)$. The translation of $\mathbf{x}$ by $h$, denoted by $T_h(\mathbf{x})\in \cC^{p-{\rm var}}([0,1], G^{\lfloor p \rfloor}(\R^d))$, is defined to be the lift of $\pi_1(\mathbf{x})+h$ to $G^{\lfloor p \rfloor}(\R^d)$. We have
\begin{enumerate}
\item There is some constant $C$ depending only on $p$ and $q$,
$$\|T_h(\mathbf{x})\|_{p-{\rm var};[0,1]}\leq C(\|\mathbf{x}\|_{p-{\rm var};[0,1]}+\|h\|_{q-{\rm var};[0,1]}).$$
\item The rough path translation $(\mathbf{x}, h)\mapsto T_h(\mathbf{x})$ as a map from
$$\cC^{p-{\rm var}}([0,1], G^{\lfloor p \rfloor}(\R^d))\times \cC^{q-{\rm var}}([0,1], \R^d)\to \cC^{p-{\rm var}}([0,1], G^{\lfloor p \rfloor}(\R^d))$$
is uniformly continuous on bounded sets.
\end{enumerate}
\end{lemma}

Now we can state the main approximating result that we need in the rough path topology on $\cC^{p-{\rm var}}([0,1], G^{\lfloor p \rfloor}(\R^d))$.

\begin{theorem}\label{th: rough-approx2}
With the notations introduced above, consider $T_h(\tilde{\mathbf{B}}^N)$. There exists a constant $\eta>0$ depending on $p, H, \|h\|_{\ch}$ and the process $B$ such that
$$\sup_N\me \lc\exp\left(\eta\|T_h(\tilde{\mathbf{B}}^N)\|^2_{p-{\rm var};[0,1]}\right)\rc < \infty.$$
Moreover, for all $q\geq 1$,
$$d_{p-{\rm var};[0,1]}(T_h(\tilde{\mathbf{B}}^N), \mathbf{h})\to 0\quad\mathrm{in}\ L^q(\mp)\ \mathrm{as}\ N\to\infty.$$ 
In the statement above, $\mathbf{h}$ is the lift of $h$ to $G^{\lfloor p \rfloor}(\R^d)$.
\end{theorem}
\begin{proof}
The first statement follows from Proposition \ref{th: rough-approx1} and Lemma \ref{th: rough-couple-conti} item (1). Moreover, note that Proposition \ref{th: rough-approx1}  and Lemma \ref{th: rough-couple-conti} item (2) imply that $d_{p-{\rm var};[0,1]}(T_h(\tilde{\mathbf{B}}^N), \mathbf{h})\to 0$ in probability, while 
$$\sup_N\me \lc \exp\left(\eta\|T_h(\tilde{\mathbf{B}}^N)\|^2_{p-{\rm var};[0,1]}\right) \rc
<\infty
$$
 implies that $d_{p-{\rm var};[0,1]}(T_h(\tilde{\mathbf{B}}^N), \mathbf{h})^q$ is uniformly integrable for any $q\geq1$.  We conclude that $d_{p-{\rm var};[0,1]}(T_h(\tilde{\mathbf{B}}^N), \mathbf{h})\to 0$ in $L^q(\mp)$ for any $q\geq 1$. This completes the proof of the second statement.
\end{proof}

We are now ready to prove the main theorem of this section.
\begin{proof}[Proof of Theorem \ref{thm:strict-positivity-intro}]
Recall that $\Phi$ is defined by \eqref{eq:skeleton}, and that the solution $X_t^x$ to equation~(\ref{eq:sde}) can be seen as $X_t^x=\Phi(\crr^{-1}B)_t$. With the definition of $T_N^h$ and that of the translation map $T_h$ in Lemma \ref{th: rough-couple-conti}, we have
$$X_t^x\circ T^h_N= \Phi(T_h(\tilde{\mathbf{B}}^N))\quad\mathrm{and} \ \ \mathbf{D}^kX_t^x\circ T^h_N=\mathbf{D}^k\Phi(T_h(\tilde{\mathbf{B}}^N)),\ \mathrm{for\ all}\ k\in\mn.$$
In the above, we consider $T_h(\tilde{\mathbf{B}}^N)$ as a geometric rough path that drives the equation for $\Phi$. Now it follows from Theorem \ref{th: rough-approx2} and the continuity of $\Phi$ and $\mathbf{D}\Phi$ in the rough path topology that
%$$X_t^x\circ T^h_N\to \Phi(h)$$
$$X_t^x\circ T^h_N\to \Phi(h),\quad \mathrm{and}\quad \mathbf{D}X_t^x\circ T^h_N\to \mathbf{D}\Phi(h)$$
%\txx{blue}{(By Theorem \ref{th: rough-approx2} and continuity, in rough path topology, of SDEs in the driving noise,  the above convergence should be at least true in probability, which is already good for our purpose.    I guess they are actually true a.s., but I can not find an exact reference for this convergence result.)}
in probability.  This shows that ({\bf{H1}}) items (1) and (2) is satisfied. 

\smallskip

For ({\bf{H1}}) item (3), recall  that $\underline{\ell}=(\mathbf{D}\Phi^1(h),\ldots,\mathbf{D}\Phi^n(h))$ and that we have set $\underline{\ell}_\alpha=(\ell_{\alpha _1},\ldots,\ell_{\alpha _k})$ for any multi-index $\alpha=(\alpha_1,\ldots,\alpha_k)\in\{1,2,\ldots,n\}^k$. By standard analysis, it suffices to show that for each multi-index $\alpha$ with  $|\alpha|=0,1,2,3,$

\begin{align*}
(R_{\underline{\ell}_\alpha,p}X^x_t)\circ T^h_N&=\int_{\{|z|\leq1\}}\left\langle(\mathbf{D}^kX^x_t)(T_z^{\underline{\ell}}B)\circ T^h_N, \ell_{\alpha _1}\otimes\cdot\cdot\cdot\otimes \ell_{\alpha _k}\right\rangle^p_{\ch ^{\otimes k}}dz\\
&=\int_{\{|z|\leq1\}}\big\langle \mathbf{D}^k\Phi(T_z^{\underline{\ell}}T_h(\tilde{\mathbf{B}}^N)), \ell_{\alpha _1}\otimes\ldots\otimes \ell_{\alpha _k}\big\rangle^p_{\ch ^{\otimes k}}dz
\end{align*}
converges to some deterministic quantity in probability.  Let
$$\hat{h}=h+\sum_{j=1}^n z^j(\mathbf{D}\Phi^j)(h).$$
The above is then reduced to show that:
$$\langle \mathbf{D}^k\Phi(T_{\hat{h}}\tilde{\mathbf{B}}^N), \ell_{\alpha _1}\otimes\ldots\otimes \ell_{\alpha _k}\rangle_{\ch ^{\otimes k}}\to \langle \mathbf{D}^k\Phi(\hat{h}), \ell_{\alpha _1}\otimes\ldots\otimes \ell_{\alpha _k}\rangle_{\ch ^{\otimes k}} $$
in probability and uniformly in $z$ for $|z|\leq 1$, which follows from Theorem \ref{th: rough-approx2}, continuity of $\mathbf{D}^k_{\ell_{\alpha _1}\ldots \ell_{\alpha _k}}\Phi(\cdot)$ in the rough path topology and the fact that $z$ takes values in a compact set.  The proof is completed.
\end{proof}

\section{Upper bounds for the density}
\label{sec:upper-bounds}

The aim of this section is to study upper bounds for the density of the solution to equation~(\ref{eq:sde}), where $B$ is a fractional Brownian motion with Hurst parameter
$H>\frac14$. Specifically, we shall prove Theorem \ref{thm:upper-bnd-density} under our elliptic Hypothesis \ref{hyp:elliptic}. 

\smallskip

Our starting point here is the integration by parts type formula given at Proposition~\ref{density}. According to this relation applied to $F=X_t^x$ and $\sigma=\{i \in \{1,\ldots,n\}: y^i \geq 0\}$, and applying inequality (\ref{Holder}) with $k=n, p=2, r=q=4$, we obtain
the following general upper bound for the density $p_t$ of $X_t^x$:
\begin{equation} \label{bound}
p_t(y) \leq c  \, \mp(\vert X^x_t-x \vert \geq \vert y-x \vert )^{1/2} \, \Vert \gamma^{-1}_t \Vert^n_{n, 2^{n+2}}  \, \Vert \bd X_t^x\Vert^n_{n, 2^{n+2}}, \; \;  \text{for all } y \in \R^n,
\end{equation}
where $\gamma_t$ denotes the Malliavin matrix of $X_t^x$.
We shall bound separately the 3 terms in relation \eqref{bound}: first, a direct application of inequality \eqref{eq:concentration-X} yields
\begin{equation}\label{eq:concentration-X-regular}
 \mp(\vert X^x_t -x\vert \geq \vert y-x \vert ) \leq \exp \left(-\frac{\vert y-x \vert^{2H+1 \wedge 2}}{c \, t^{2H} }  \right).
\end{equation}
Next, we prove that there exist constants $c_3$ and $c_4$ such that for all $m \in \mathbb{N}$ and $p>1$,
\begin{align} 
 \label{derivative}
 \Vert \bd X_t^x\Vert_{m, p} &\leq c_3 \,  t^{H}
 \\ \label{gamma}
\Vert \gamma^{-1}_t \Vert_{m, p} &\leq c_4 \, t^{-2H}. 
\end{align}
Plugging relations \eqref{eq:concentration-X-regular}-\eqref{gamma} into (\ref{bound}), this will conclude  the proof of Theorem \ref{thm:upper-bnd-density}.

We start with the estimate \eqref{derivative}.

\begin{lemma}\label{th:DX mp norm}
Let $H>\frac{1}{4}$. Denote by $X_t^x$ the solution to equation \textnormal{(\ref{eq:sde})}.  One has
$$\|\bd X_t^x\|_{m,p}\leq c_{m,p} t^{H},
$$
for some constant $c_{m,p}>0$. 
\end{lemma}

\begin{proof}
%\noindent Do we have an embedding like $\ch\hookrightarrow\cac^{p-var}$. Since we are trying to bound $\|\bd^nX\|_{\ch^{\otimes n}}$. Assuming (\ref{eq:high-order-bound2}) and the correct embedding, to show that $\bd^nX$ is bounded in $\|\cdot\|_{\ch^{\otimes n}}$, I guess we have something like

We use a method by Inahama \cite{Inahama} to which we refer for more details. For simplicity, we assume $V_0=0$, and first show for $m=1,2$. The case $V_0 \neq 0$ is treated similarly. Recall $\mathbf{J}$ is the Jacobian process. %For any element $h, k\in\mathcal{\bar H}$, we have
%$$\bd_hX_t=\mathbf{J}_t\int_0^t\mathbf{K}_sV(X_s^x)dh_s,$$
%and
%\begin{align*}
%\bd^2_{h,k}X_t^x=&\mathbf{J}_t\int_0^t\mathbf{K}_s\{D^2V(X_s^x)\langle \bd_hX_s^x, \bd_kX_s^x, dB_s\rangle\\
%&\quad\quad +DV(X_s^x)\langle \bd_k, dh_s\rangle+DV(X_t^x)\langle \bd_hX_s^x, dk_s\rangle.
%\end{align*}

Let $\hat{B}=(\hat{B}_1,...,\hat{B}_d)$ be an independent copy of $B$ and consider $2d$-dimensional fractional Brownian motion $(B, \hat{B}).$ The expectation with respect to $B$ and $\hat{B}$ are denoted by $\me$ and $\hat{\me}$.
Set
$$\Xi_1(t)=\mathbf{J}_t\int_0^t\mathbf{J}^{-1}_sV(X_s^x)d\hat{B}_s,$$
and
\begin{align*}
\Xi_2(t)=&\mathbf{J}_t\int_0^t\mathbf{J}^{-1}_s\{D^2V(X_s^x)\langle \Xi_1(s), \Xi_1(s), dB_s\rangle +2DV(X_t^x)\langle \Xi_1(s), d\hat{B}_s\rangle\}.
\end{align*}
More generally, we can construct a $\Xi_m$ by induction (see \cite{Inahama}). %Now we fixed a sample path $w$ of $B$. It is clear that $\Xi_m(w,\cdot)_t$ belongs to the inhomogeneous Wiener chaos of order $m$ for $m=1,2$. Moreover denote $\hat{\bd}$ the derivative with respect to $\hat{B}$, we have
%$$\hat{\bd}_h\Xi_1(w,\hat{B})_t=\bd_hX_t^x(w),\quad \mathrm{and}\ \ \hat{\bd}_{h,k}\Xi_2(w,\hat{B})_t=2\bd^2_{h,k}X_t^x(w).$$
%By the fact that all $\mathbb{D}^{k,p}$-norms (k=0,1,...) are equivalent on fixed inhomogeneous Wiener chaos, we obtain
%$$\|\bd X_t(w)\|_{\mathcal{H}\otimes\mr^n}=\hat{\me}\big(\|\hat{\bd}\Xi_1(w,\cdot)_t\|^2_{\mathcal{H}\otimes\mr^n}\big)^{1/2}\lesssim\|\Xi_1(w,\cdot)_t\|_{1,2}\lesssim\|\Xi_1(w,\cdot)_t\|_2,$$
%$$\|\bd^2 X_t(w)\|_{\mathcal{H}\otimes\mathcal{H}\otimes\mr^n}=2\hat{\me}\big(\|\hat{\bd}^2\Xi_2(w,\cdot)_t\|^2_{\mathcal{H}\otimes\mathcal{H}\otimes\mr^n}\big)^{1/2}\lesssim2\|\Xi_2(w,\cdot)_t\|_{2,2}\lesssim\|\Xi_2(w,\cdot)_t\|_2.$$
%Next we take expectation $\me$. 
Then one can show that,
$$\|\bd X_t^x\|_{\mathcal{H}\otimes\mr^n} \le C \hat{\me}(\vert\Xi_1(t)\vert^2)^{1/2},$$
$$\|\bd^2 X_t^x\|_{\mathcal{H}\otimes\mathcal{H}\otimes\mr^n}\le C \hat{\me} (\vert \Xi_2(t) \vert^2)^{1/2}.$$
We now estimate $\Xi_1$ and $\Xi_2$ by using rough paths theory. Let 
\begin{align}\label{rough M}
M=(B, \hat{B}, X^x, \mathbf{J}, \mathbf{J}^{-1}).
\end{align}
 This is a $p$-rough path, $p>1/H$. The integral $\int \mathbf{J}^{-1}_sV(X_s^x)d\hat{B}_s$ is a rough integral of the type $\int f(M) d\mathbf{M}$, where $f$ has a polynomial growth. We deduce  the bound
\[
| \Xi_1 (t) -\Xi_1(s)| \le C (1+\| \mathbf{M}\|_{p-var,[0,1]})^r \| \mathbf{M}\|_{p-var,[s,t]}.
\]
We now estimate  $\|\mathbf{M}\|_{p-var,[s,t]}$. Denote by $D(t)$ a subdivision of the interval $[0,t]$. Define
\begin{align*}
\mathcal{M}_{\alpha,t,p}=\sup_{D(t)=(t_i); \| \mathbf{B}\|^p_{p-var,[t_i,t_{i+1}]}\leq \alpha}\sum_{i: t_i\in D(t)} \|\mathbf{B}\|^p_{p-var,[t_i,t_{i+1}]}.
\end{align*}
Then the Jacobian $\bj$ satisfies the following growth-bound:
\begin{align*}
\|\bj\|_{p-{\rm var}; [0,t]} +\|\bj^{-1}\|_{p-{\rm var}; [0,t]} \leq C \, \| \mathbf{B}\|_{p-var,[0,t]}\exp\left({C \mathcal{M}_{\alpha,t,p}}\right).
\end{align*}
For some constant $c$ (cf Proposition 4.11 in \cite{CLL}), we have $M_{\alpha,t,p}\leq c(N_{\alpha,t,p}+1)\alpha.$
Hence we obtain a bound for $\Vert \bj \Vert_{p-{\rm var};[0,t]}$ of the form:
\begin{equation}\label{eq:upp-bnd-J-with-B-p-var}
\|\bj\|_{p-{\rm var}; [0,t]} +\|\bj^{-1}\|_{p-{\rm var}; [0,t]} \leq C \, \| \mathbf{B}\|_{p-var,[0,t]}\exp\left({C N_{\alpha,t,p}}\right).
\end{equation}
We eventually  deduce a bound of the form
\[
| \Xi_1 (t) | \le C (1+\| \mathbf{M}\|_{p-var,[0,1]})^r (\| \mathbf{B}\|_{p-var,[0,t]}+\| \hat{\mathbf{B}}\|_{p-var,[0,t]})\exp\left({C N_{\alpha,t,p}}\right).
\]
By scaling we have $\| \mathbf{B}\|_{p-var,[0,t]}+\| \hat{\mathbf{B}}\|_{p-var,[0,t]}\stackrel{law}{=} t^H (\| \mathbf{B}\|_{p-var,[0,1]}+\| \hat{\mathbf{B}}\|_{p-var,[0,1]})$.
The proof is thus completed for the case $m=1$. In the same way, we estimate $\Xi_2$ as a rough integral of the type $\int \phi (M_1) d\mathbf{M}_1$ where $\phi$ has polynomial growth and $M_1$ is the rough path
\[
M_1 =(B, \hat{B}, X^x, \mathbf{J}, \mathbf{J}^{-1},\Xi_1)
\]
Arguing as before and using previous estimates we obtain then a bound of the same type:
\[
| \Xi_2 (t) | \le C (1+\| \mathbf{M}\|_{p-var,[0,1]})^r (\| \mathbf{B}\|_{p-var,[0,t]}+\| \hat{\mathbf{B}}\|_{p-var,[0,t]})\exp\left({C N_{\alpha,t,p}}\right).
\]
Higher order Malliavin derivative are treated similarly.
\end{proof}

\subsection{The regular case}
\label{sec:upper-bnd-regular}

In this section we treat the case where $B$ is a fractional Brownian motion with Hurst parameter $H >\frac12$. In this situation, the stochastic integral in (\ref{eq:sde})
can be seen as a Young integral instead of the general rough paths type integral invoked at Proposition \ref{prop:moments-sdes-rough}. Moreover, the proof of our upper bound can be summarized as follows:

\begin{proof}[Proof of Theorem \ref{thm:upper-bnd-density} in the regular case]
Recall that under the elliptic Hypothesis \ref{hyp:elliptic} and assuming $H>1/2$ we wish to show that
\begin{equation}\label{eq:exp-bound-regular}
p_t(y) \leq c_{2} \, t^{-nH} \exp \left(-\frac{\vert y-x \vert^2}{c_{1} t^{2H} }  \right), \; \;  \text{for all } y \in \R^n.
\end{equation}

\smallskip

The proof of  (\ref{derivative}) is treated in a uniform way for both the regular and irregular cases in Lemma \ref{th:DX mp norm}.  Hence let let us concentrate here on the proof of (\ref{gamma})  for $0 < t \le 1$. Let 
\[
C_t= \int_0^t \int_0^t \mathbf{J}_{u}^{-1} V(X_u^x)V(X_u^x)^*(\mathbf{J}_{u}^{-1} )^* |u-v|^{2H-2} du dv.
\]
Our bound \eqref{gamma} is now reduced to prove that 
\begin{equation}\label{eq:low-bnd-Sigma-t}
y^{*} C^{-1}_t y \le M t^{-2H} \, |y|^{2}, 
\quad \text{for} \quad
y\in\R^{n},
\end{equation}
for a given random variable $M$ admitting moments of any order. To this aim, notice first that
\begin{equation*}
y^{*} C_t y = 
\int_0^1 \int_0^1 
\lla f_{u} , \, f_{v} \rra_{\R^{n}}
|u-v|^{2H-2} \, du dv,
\quad \text{with} \quad
f_{u}\equiv \mathbf{1}_{[0,t]}(u) V(X_{u}^x)^*(\mathbf{J}_{ u}^{-1} )^* y.
\end{equation*}
Furthermore, thanks to the interpolation inequality  of Proposition \ref{interpolation} applied with $\gamma >H-\frac{1}{2}$, we have 
\begin{equation}\label{eq:interpolation-H-L2}
\int_0^1 \int_0^1 \langle f_{u}, f_{v} \rangle |u-v|^{2H-2}dudv  \ge ct^{2H} \frac{\min_{[0,1]} | f|^4 }{\|f \|_\infty^2+\| f \|^2_{\gamma}},
\end{equation}
where $\| f \|_{\gamma}$ is the $\gamma$-H\"older norm of $f$ on the interval $[0,1]$ as defined at \eqref{eq:def-holder-norms}. Furthermore, since the uniform ellipticity condition $| V(x) y |^2 \ge \lambda | y |^2$ holds true, it is readily checked that
\begin{equation}\label{eq:bnd-f}
|f_v|^{2} \ge \la \, |\mathbf{J}_{ v}^{-1} y |^{2}\ge \la \, \|\mathbf{J}_{ v}\|^{-2} | y |^{2},
\quad \text{and} \quad
\|f \|_\infty+ \| f \|_{\gamma} \le c\, (1+\|X\|_{\gamma}) (1+\|\bj^{-1}\|_{\gamma}) | y |.
\end{equation}
Plugging these relations into \eqref{eq:interpolation-H-L2} we deduce that for every  $y \in \mathbb{R}^n$,
\[
 y^*C^{-1}_t y \le c t^{-2H}  \,  (1+\|X\|_{\gamma})^{2} (1+\|\bj^{-1}\|_{\gamma})^{2} (1+ \| \bj\|_{\gamma})^4
\,  \vert y\vert^2,
\]
from which \eqref{eq:low-bnd-Sigma-t}, and thus (\ref{gamma}), are easily deduced.

\smallskip

For the bound of Malliavin derivatives of $\gamma_t^{-1}$, note that we have 
\begin{align}\label{D gamma}
\bd(\gamma_t^{-1})^{ij}=-\sum_{k,l=1}^d(\gamma_t^{-1})^{ik}(\gamma_t^{-1})^{lj}\bd\gamma_t^{kl}.
\end{align}
Therefore
\begin{align*}
\|\bd(\gamma_t^{-1})^{ij}\|_\ch\leq |(\gamma_t^{-1})^{ik}(\gamma_t^{-1})^{lj}|(\|\bd X_t\|_\ch+\|\bd^2X_t\|_{\ch^{\otimes 2}})^2.
\end{align*}
Together with the estimates for $\|\bd X_t\|_{m,p}$ and $\|\gamma_t^{-1}\|_{p}$ that have been established above, we have
$$\|\bd\gamma^{-1}_t\|_{1,p}\leq c \, t^{-2H}.
$$
Similarly, by using equation (\ref{D gamma}) repeatedly, we  conclude that  for each $m\in\mathbb{N}$ and $p>1$ there exists a constant $c_{m,p}$  such that 
$$\|\bd\gamma^{-1}_t\|_{m,p}\leq c_{m,p} \, t^{-2H}.$$
\end{proof}

\subsection{The irregular case }

The aim of this section is to extend the results of the last section to the case where $B$ is a fractional Brownian motion with Hurst parameter $H \in(\frac14,\frac12)$.
For this, tools of rough paths theory are required to obtain the  sub-Gaussian bound \eqref{eq:exp-bound-irregular}.  

From the discussion above it is clear that, in order to conclude the correct asymptotic behavior (as $t\downarrow 0$) in the upper bound for the density function, we need to establish (\ref{derivative})  and (\ref{gamma}) for the irregular case.   We first prove (\ref{derivative}) in both the regular and irregular cases.

The counterpart of (\ref{gamma}) in the rough case is the content of the following lemma. 

\begin{lemma}\label{th:gamma 0p norm}
Let $\frac{1}{4}<H<\frac{1}{2}$. Denote by $X_t^x$ the solution to equation \textnormal{(\ref{eq:sde})}, and $\gamma_t$ the Malliavin matrix of $X_t^x$. Under Hypothesis \textnormal{\ref{hyp:elliptic}}, there exists a constant $c_{m,p}>0$ such that for all $t\in(0,1]$ one has
$$\|\gamma^{-1}_t\|_{m,p}\leq c_{m,p} \, t^{-2H}.
$$
\end{lemma}
\begin{proof}
We first prove the lemma for $m=0$.   As before the bound we want to prove is reduced to prove that 
\begin{equation}\label{eq:low-bnd-Sigma-t2}
y^{*} C^{-1}_t y \le M t^{-2H} \, |y|^{2}, 
\quad \text{for} \quad
y\in\R^{n},
\end{equation}
for a given random variable $M$ admitting moments of any order, where, again, $C$ is the reduced Malliavin matrix defined by
\begin{equation*}
y^{*} C_t y = \| f \|_\ch^2
\quad \text{with} \quad
f_{u}\equiv \mathbf{1}_{[0,t]}(u) V(X_{u}^x)^*(\mathbf{J}_{ u}^{-1} )^* y.
\end{equation*}
From the inequality  of Proposition \ref{interpolation} and the uniform ellipticity assumption, we have thus,
\[
 y^*C^{-1}_t y \le c t^{-2H}   (1+ \| \bj\|_{\gamma})^2
\,  \vert y\vert^2,
\]

This yields the claimed result when $m=0$.

\smallskip

For $m\geq1$, note that by Lemma \ref{th:DX mp norm} and what we have just proved,  there exist constants $c_{m,p}$ and $c_{p}$ such that $ \Vert \bd X_t^x\Vert_{m,p} \leq c_{m,p}t^H$ and $\Vert \gamma^{-1}_t \Vert_{p} \leq c_{p} t^{-2H}$.
Putting this together with relation \eqref{D gamma} and along the same lines as in the smooth case, we can conclude that there exists a constant $c_{m,p}$ such that
$\|\bd\gamma^{-1}_t\|_{m,p}\leq c_{m,p} t^{-2H},$
for all $m\in\mathbb{N}$ and $p> 1$. 

\end{proof}

We can now prove our sub-Gaussian upper bound for the density $p_t(\cdot)$ of $X_t^x$ in the rough case:

\begin{proof}[Proof of Theorem \ref{thm:upper-bnd-density} in the irregular case]
Owing to inequality \eqref{eq:concentration-X}, we have
\begin{equation*}
 \mp(\vert X^x_t -x\vert \geq \vert y-x \vert ) \leq \exp \left(-\frac{2\vert y-x \vert^{2H+1}}{c \, t^{2H} }  \right).
\end{equation*}
Now the proof follows from (\ref{bound}),  and Lemmas \ref{th:DX mp norm} and \ref{th:gamma 0p norm} just as in the smooth case.

\end{proof}

\begin{remark}
In order to prove Theorem \ref{thm:upper-bnd-density}, we could also have used the new expression for the density of a non-degenerate random vector obtained recently by Bally and Caramellino in \cite{BC}. This expression involves the Poisson kernel, and only requires the random vector to be twice differentiable in the Malliavin sense, in comparison with Proposition \ref{density} where higher derivatives are involved. However, we have not included the details of this strategy here, since it yields some slightly non optimal coefficients in relation \eqref{eq:exp-bound-irregular}.
\end{remark}

\section{Hitting probabilities and capacities}
\label{sec:hitting-probabilities}

We now turn to the evaluation of hitting probabilities for our differential system \eqref{eq:sde-intro}, that is the proof of relation \eqref{eq:bnd-hitting-regular} in Theorem \ref{thm:hitting-capacity-X}. It should be noticed that the upper and lower bounds in those relations require a different methodology, and this is why they shall be studied in two separate sections.

\subsection{Lower bounds on hitting probabilities}
As established in  \cite[Theorem 2.1]{DKN}, the lower bound in \eqref{eq:bnd-hitting-regular} can be derived from a general result for the hitting probabilities of a continuous stochastic process in terms of its finite-dimensional density functions. We shall prove this general relation in our fBm context for the sake of clarity.

\smallskip

Specifically, suppose that $(u_{t}, t\geq 0)$ is a continuous stochastic process in $\R^n$, such that the random vector $(u_{t}, u_{s})$  
has a joint probability density function $p_{s,t}(\cdot\,,\cdot)$,
for all $s,t>0$ such that $s \neq t$. As in the previous sections, we will also denote by $p_t(\cdot)$ the density of $u_{t}$, for all $t>0$. We work under the following set of hypotheses:

\smallskip

\noindent\textbf{(A1)}
For all $0<a<b$ and $M>0$, there exists a positive constant $C
        = C(a,b,M,n)$ such that for all $z\in[-M\,,M]^n$,
        \begin{equation*}
             \int_a^b p_{t}(z) dt \geq C.
        \end{equation*}

\smallskip

\noindent\textbf{(A2)}
There exist $\beta>0$, $H \in (0,1)$ and $p>\beta$ such that for all
        $0<a<b$, $M>0$, one can find a constant $c = c(a,b,\beta,H,M,n,p)>0$
        such that for all $s,t \in [a,b]$
        with $s \neq t$,
        and for every $z_1,z_2\in[-M\,,M]^n$,
\begin{equation*} 
p_{s,t}(z_1, z_2) \leq  \frac{c}{\vert t-s\vert^{H \beta}} \left(
            \frac{\vert t-s \vert^H}{\vert x-y \vert} \wedge 1\right)^p.
\end{equation*}

\smallskip

With these assumptions in hand, our general result on hitting probabilities is the following:
\begin{theorem}\label{thm:cap:LB}
    Suppose {\bf (A1)} and {\bf (A2)} are met, and fix $0<a<b$ and $M >0$. 
Then there exists a strictly positive constant $c=c(a,b,\beta,H,M,n)$ such that for all compact sets $A\subseteq[-M\,,M]^n$,
            \begin{equation}\label{eq:low-bnd-cap-general}
                \mathbb{P}(u([a,b]) \cap A
                \neq\varnothing) \geq c\, 
                \textnormal{Cap}_{\alpha}(A).
            \end{equation}
where $\alpha=\beta-\frac{1}{H}$.
\end{theorem}

\begin{proof}
We start by proving a technical lemma that gives
the relationship between the upper bound in {\bf (A2)} and the
Newtonian kernel ${\rm K}_{\alpha}$ defined by (\ref{kernel}).
\begin{lemma} \label{le}
Let $N>0$, $\beta>0$, $p>\beta$, $0 \leq a < b$, and $H \in (0,1)$ be fixed.
Then, there exists a positive constant $C=C(a,b,\beta,N,H,p)$ such 
that for all $r \in [0,N]$
\begin{equation}\label{eq:bnd-intg-expo}
\int_a^b  \int_a^b \frac{1}{\vert t-s \vert^{H \beta}} \left(
            \frac{\vert t-s\vert^H}{r} \wedge 1\right)^p ds dt\leq C \, 
            {\rm K}_{\alpha}( r ),
\end{equation}
where $\alpha=\beta-\frac{1}{H}$.
\end{lemma}

\begin{proof}
Fix $r \in [0, N]$ and use
the change of variables $u=t-s$, to see that
\begin{equation*}
\int_a^b \int_a^b   \frac{1}{\vert t-s \vert^{H \beta}} \left(
            \frac{\vert t-s\vert^H}{r} \wedge 1\right)^p ds dt
            \leq 2(b-a) \int_0^{b-a}  u^{-H \beta} \left(
            \frac{u^H}{r} \wedge 1\right)^p  du.
\end{equation*}
Next, the change of variables $v= \frac{u^H}{r}$ implies that the right hand side
equals
\begin{equation*}
C r^{-\alpha} F(m),
\quad\text{where}\quad
F(m):=\int_0^{m} 
v^{-\beta-1+\frac{1}{H}} (v \wedge 1)^p dv,
\end{equation*}
with the notation $m:=\frac{(b-a)^{H}}{r}$.
Observe that $m \geq m_1:=\frac{(b-a)^{H}}{N}>0$.
Hence, we can split $F(m)$ into $F(m)=F(m_1)+[F(m)-F(m_1)]$. 
Now clearly, we have $F(m_1) \leq c$, and if $\beta \neq \frac{1}{H}$, then
$$
F(m)-F(m_1)\leq \frac{m^{\frac{1}{H}-\beta}-m_1^{\frac{1}{H}-\beta}}{\frac{1}{H}-\beta}.
$$
Hence, if $\beta > \frac{1}{H}$, we get $F(m)-F(m_1) \leq c$;
if $\beta < \frac{1}{H}$, then 
$F(m)-F(m_1) \leq C r^{\beta-\frac{1}{H}}$; and if $\beta = \frac{1}{H}$, some similar elementary computations show that
$$
F(m)-F(m_1) \leq C \log(m)=c+ c' \log \left( \frac{1}{r} \right).
$$
Therefore, putting together these considerations we conclude the proof of relation \eqref{eq:bnd-intg-expo}, provided that the constant
$N_0$ in (\ref{kernel}) is sufficiently large.

\end{proof}

Let us now go back to the proof of Theorem \ref{thm:cap:LB}: fix a compact set $A\subseteq[-M\,,M]^n$ and observe that whenever $\textnormal{Cap}_{\alpha}(A)=0$, inequality \eqref{eq:low-bnd-cap-general} is trivially satisfied. In the remainder of the proof we thus assume $\textnormal{Cap}_{\alpha}(A)>0$. In particular,
this implies that $A \neq \varnothing$. We now consider three different cases:

\smallskip

\noindent
{\it Case 1:} $\beta<\frac{1}{H}$.  Then $\alpha<0$ and thus $\textnormal{Cap}_{\alpha}(A)=1$. Therefore,
it suffices to prove that there exists a positive constant $c=c(a,b,M,H, \beta,n)$ such that
\begin{equation} \label{step1}
\mathbb{P}(u([a,b]) \cap A
                \neq\varnothing) \geq c.
                \end{equation}
Towards this aim, for all $\epsilon \in (0,1)$ and $z \in \R^n$, consider the random variable
$$
J_{\epsilon}(z)=\frac{1}{(2\epsilon)^n} \int_a^b  {\bf 1}_{\tilde{B}(z, \epsilon)}(u_{t}) dt,
$$
where $\tilde{B}(z, \epsilon)=\{y \in \R^n: \vert z-y \vert < \epsilon\}$
and $\vert z \vert=\max_{1 \leq i \leq n} \vert z_i \vert$.
Assume now that $z \in A$. Our first aim is to prove that $\mathbb{P} ( J_{\epsilon}(z) >0) \geq C$, for a strictly positive constant $C$ independent of $\epsilon$. Indeed, Hypothesis {\bf (A1)} implies that there exists a positive constant $C(a,b,M,H,n)$ such that for all $\epsilon \in (0,1)$,
$$
\mathbb{E} \left[ J_{\epsilon}(z) \right]=\frac{1}{(2\epsilon)^n} \int_a^b
\int_{\tilde{B}(z, \epsilon)} p_t(v) dv dt \geq C.
$$
On the other hand, Hypothesis {\bf (A2)}  and Lemma \ref{le} imply that
there exists a positive constant $C(a,b,M,H,\beta, n)$ such that for all $\epsilon \in (0,1)$,
\begin{equation*} 
\begin{split}
\mathbb{E} \left[ J^2_{\epsilon}(z) \right]
&=\frac{1}{(2\epsilon)^{2n}} \int_a^b \int_a^b \int_{\tilde{B}(z, \epsilon)}
\int_{\tilde{B}(z, \epsilon)} p_{s,t}(z_1,z_2) dz_1 dz_2 ds dt  \\
&\le  \frac{c}{(2\epsilon)^{2n}} \int_{\tilde{B}(z, \epsilon)} \int_{\tilde{B}(z, \epsilon)}
K_{\alpha}(z_2-z_1) \, dz_1 dz_2
\leq c,
\end{split}
\end{equation*}
where the last inequality is due to the fact that $K_{\alpha}\equiv 1$ whenever $\alpha<0$. Therefore, from Paley-Zygmund inequality (cf. \cite[(2.26)]{DKN}), we conclude that
\begin{equation}\label{eq:low-bnd-P-J-eps}
\mathbb{P} \left( J_{\epsilon}(z) >0\right) \geq 
\frac{\mathbb{E} \left[ J_{\epsilon}(z) \right]^2}{\mathbb{E} \left[ J^2_{\epsilon}(z) \right]} \geq C,
\end{equation}
where $C$ is independent of $\epsilon$. Moreover, the left-hand side of \eqref{eq:low-bnd-P-J-eps} is bounded above by $\mathbb{P}(u([a,b]) \cap A_{\epsilon}
                \neq\varnothing)$, where 
                $A_{\epsilon}$ denotes the closed $\epsilon$-enlargement
                of $A$. Then we let $\epsilon \downarrow 0$ and use
                the continuity of the trajectories of $u$ to conclude that
                (\ref{step1}) holds true.

\smallskip

\noindent
{\it Case 2:} $\beta>\frac{1}{H}$. For all $\epsilon \in (0,1)$ and $\mu \in \mathcal{P}(A)$, consider the random variable
$$
J_{\epsilon}(\mu)=\frac{1}{(2\epsilon)^n} \int_{\R^n} \int_a^b  {\bf 1}_{\tilde{B}(z, \epsilon)}(u_{t}) \, dt \, \mu(dz).
$$
Then {\bf (A1)} implies the existence of a positive constant $C(a,b,M,H,n)$ such that
$$
\mathbb{E} \left[ J_{\epsilon}(\mu) \right] \geq C.
$$
In order to estimate the second moment of  $J_{\epsilon}(\mu)$,
we consider the function $$g_{\epsilon}(z)=(2\epsilon)^{-n}  {\bf 1}_{\tilde{B}(0, \epsilon)}(z),$$
so that we can write
$$
J_{\epsilon}(\mu)= \int_a^b  [g_{\epsilon} \ast \mu](u_{t}) dt .
$$
It is readily checked that
\begin{equation*}
\EE\lc J^2_{\epsilon}(\mu) \rc
= \int_{\R^{n}\times\R^{n}} [g_{\epsilon} \ast \mu](z_1) \, [g_{\epsilon} \ast \mu](z_2)
\lp \int_{[a,b]^{2}} p_{s,t}(z_1,z_2) \, ds dt \rp \, dz_1 dz_2,
\end{equation*}
and thus, owing to Hypothesis {\bf (A2)}  and Lemma \ref{le} we obtain that
there exists a positive constant $c=c(a,b,M,H,\beta,n)$ such that
$$
\mathbb{E} \left[ J^2_{\epsilon}(\mu) \right]
\leq c \, \mathcal{E}_{\alpha}(g_{\epsilon} \ast \mu),
$$
where we recall that the energy functional $\ce_{\alpha}$ has been defined by relation \eqref{eq:def-E-beta-mu}.
We now choose $\mu \in \mathcal{P}(A)$ such that $\mathcal{E}_{\alpha}(\mu)
\leq \frac{2}{\textnormal{Cap}_{\alpha}(A)}$. We also recall that, thanks to the general result~\cite[Theorem B.1]{DKN}, we have $\mathcal{E}_{\alpha}(g_{\epsilon} \ast \mu)\le \mathcal{E}_{\alpha}( \mu)$ for all $\epsilon \in (0,1)$. We thus obtain that:
$$
\mathbb{E} \left[ J^2_{\epsilon}(\mu) \right]
\leq  \frac{2c}{\textnormal{Cap}_{\alpha}(A)}.
$$
Therefore, from Cauchy-Schwarz inequality, we conclude that
\begin{equation}\label{eq:low-bnd-P-J-mu-eps}
\mathbb{P} \left[ J_{\epsilon}(\mu) >0\right] \geq 
\frac{\mathbb{E} \left[ J_{\epsilon}(\mu) \right]^2}{\mathbb{E} \left[ J^2_{\epsilon}(\mu) \right]} \geq  \frac{c}{\textnormal{Cap}_{\alpha}(A)},
\end{equation}
where the positive constant $c$ is independent of $\mu$.
As for the first case, the left-hand side of~\eqref{eq:low-bnd-P-J-mu-eps} is upper bounded by $\mathbb{P}(u([a,b]) \cap A_{\epsilon}
                \neq\varnothing)$, where 
                $A_{\epsilon}$ denotes the closed $\epsilon$-enlargement
                of $A$. Then we let $\epsilon \downarrow 0$ and use
                the continuity of the trajectories of $u$ to assert that \eqref{eq:low-bnd-cap-general} holds true in our Case 2.

\smallskip

\noindent
{\it Case 3:} $\beta=\frac{1}{H}$. This case follows exactly along the same lines as Case 2, except for the fact that we appeal to \cite[Theorem B.2]{DKN} instead of  \cite[Theorem B.1]{DKN}.

\end{proof}

From the definition of capacity and as a consequence of Theorem \ref{thm:cap:LB}, we have the following result
on hitting points for the process $u$.
\begin{corollary} \label{cor11}
Under the hypotheses of Theorem \ref{thm:cap:LB2}, if $\beta <\frac{1}{H}$, the process $u$ hits points
in $\R^n$ with strictly positive probability, that is, 
$$
  \mathbb{P}(\exists \, t > 0 : u_{t}=x)>0 \quad \text{ for all } x \in \R^n.
$$
\end{corollary}

\begin{proof}
Observe that we have $\alpha<0$ whenever $\beta<\frac{1}{H}$. Thus, in this case, $\textnormal{Cap}_{\alpha}(\{x\})=1$ for any $x\in\R^{n}$. On the other hand, we write
$(0, \infty)=\cup_{m \in \mathbb{N}} [\frac{1}{m}, m]$.
Then by Theorem \ref{thm:cap:LB}, for all $m \geq 1$,
there is $c>0$ depending on $m$ such that
$$
  \mathbb{P} \left(\exists \, t \in \left[\frac{1}{m}, m\right]: u_t=x \right) \geq c \,
 \textnormal{Cap}_{\alpha}(\{x\})=c>0.
$$
Since this holds for all $m$, the desired result holds.
\end{proof}

\subsection{Bivariate density bound}
We will now apply apply the general result of Theorem \ref{thm:cap:LB} to the $n$-dimensional process solution to  equation \textnormal{(\ref{eq:sde})}. In order to achieve this goal, the main remaining technical difficulty consists in proving the upper bound for the bivariate density stated at condition {\bf (A2)}. In this case, our strategy hinges on conditional integration by parts in the Malliavin calculus sense, which turns out to be much easier to express in terms of the underlying Wiener process $W$ induced by the Volterra representation \eqref{eq:volterra-representation}. This idea is also present in \cite{BKT}, and it forces us to introduce some additional notation.

\smallskip

We shall manipulate Malliavin derivatives with respect to both $B$ and $W$. In order to distinguish them, the Malliavin derivatives with respect to $W$ will be denoted by $D$ and the Sobolev spaces by $D^{k,p}$.
The relationship between the two kinds of derivatives are recalled in the following:

\begin{proposition} \label{prop:relation-D-DW}
Let ${D}^{1,2}$ be the Malliavin-Sobolev space corresponding to
the Wiener process $W$. Then $\D^{1,2}=(K^{*})^{-1}{D}^{1,2}$ and for
any $F\in {D}^{1,2}$ we have ${\mathrm D} F = K^{*} \bd F$
whenever both members of the relation are well defined. 
\end{proposition}

In particular,  we can compute the Malliavin derivative  of $(X_t^x)_{t \ge 0}$ with respect to $W$ as follows:

\begin{proposition}\label{prop:deriv-sdeW} 
Let $X^x$ be the solution to equation (\ref{eq:sde}) and suppose the $V_i$'s satisfy Hypothesis \ref{hyp:regularity-V}. Then
for every $i=1,\ldots,n$, $t>0$, and $x \in \mathbb{R}^n$, we have $X_t^{x,i} \in
D^{\infty}$ and
\begin{equation*}
D^j_s X_t^{x}= \mathbf{J}_{t} Q^j_{st}, \quad j=1,\ldots,d, \quad 
0\leq s \leq t,
\end{equation*}
where $D^j_s X^{x,i}_t $ is the $j$-th component of
$D_s X^{x,i}_t$, $\mathbf{J}_{t}=\partial_{x} X^x_t$ is defined at Proposition \ref{prop:deriv-sde},   and
 \begin{align}\label{malliavin-W2}
 Q^j_{st}=
 \begin{cases}
 \int_s^t \partial_u K(u,s) \mathbf{J}^{-1}_{u} V_j(X_u) du, \quad H >1/2 \\
K(t,s)  \mathbf{J}^{-1}_{s} V_j(X_s) + \int_s^t  \lp  \mathbf{J}^{-1}_{r} V_j(X_r)- \mathbf{J}^{-1}_{s} V_j(X_s) \rp  \, \partial_r K(r,s)  dr, \quad H \le 1/2.
\end{cases}
\end{align}
\end{proposition}

Recall that we have chosen to express our conditional integration by parts formula in terms of the underlying Wiener process $W$, because projections on subspaces are easier to describe in a $L^{2}$ type setting. 
We now state this conditional integration by parts formula: For a random variable $F$ and $t\in [0,1]$, let $\|F\|_{m,p,t}$ and
  $\Gamma_{F,t}$ be the quantities defined (for $m\ge 0$, $p>0$) by:
\begin{equation}\label{eq:def-conditional-sobolev}
\|F\|_{m,p,t}=\left( \mathbb{E}_t\left[
F^{p}\right]
+\sum_{j=1}^m \mathbb{E}_t\left[ \left\| {D}^j F\right\|
_{(L^{2}_{t})^{\otimes j}}^{p}\right] \right) ^{\frac{1}{p}},
\text{ and }
{\Gamma}_{F,t}=
\lp\langle D F^i, D F^j\rangle_{L^{2}_{t}}\rp_{1\leq i,\,j\leq n},
\end{equation}
where we have set $L^{2}_{t}\equiv L^{2}([t,1])$ and $\mathbb{E}_t=\mathbb{E}( \cdot | \mathcal{F}_t)$.
With this notation in hand, the following formula is borrowed from
\cite[Proposition 2.1.4]{Nu06}:

\begin{proposition} \label{prop:int-parts-cond-W}
Fix $k\geq 1$. Let $F,\,Z_s,\,G\in({D}^{\infty})^n$ be three
random vectors where $Z_s\in\cf_s$-measurable and 
$(\det_{{\Gamma}_{F+Z_s}})^{-1}$ has finite moments of
all orders. Let $g\in\cac_p^\infty(\mr^d)$. Then, for any
multi-index
$\alpha=(\alpha_1,\,\ldots,\,\,\alpha_k)\in\{1,\,\ldots,\,n\}^k$,
there exists a r.v. ${H}^s_\alpha(F,G)\in\cap_{p\geq
1}\cap_{m\geq 0} {D}^{m,p}$ such that
\begin{equation}
\label{eq:int-parts-cond-W}
\EE_s\lc(\partial_\alpha g)(F+Z_s) \, G\rc
=
\EE_s\lc g(F+Z_s) \, {H}_\alpha^s(F,G)|\rc,
\end{equation}
where ${H}_\alpha^s(F,G)$ is
recursively defined by
\begin{equation*}
{H}_{(i)}^s(F,G)=\sum_{j=1}^n {\delta}_s\lp
G\lp{{\Gamma}}_{F,s}^{-1}\rp_{ij}D F^j\rp,
\quad
{H}_{\alpha}^s(F,G)={H}^s_{(\alpha_k)}(F,{H}^s_{(\alpha_1,\,\ldots,\,\alpha_{k-1})}(F,G)).
\end{equation*}
Here ${\delta}_s$ denotes the Skorohod integral with respect
to the Wiener process $W$ on the interval $[s,1]$.
Furthermore, the following norm estimates hold true:
\begin{equation} \label{eq:bnd-H-alpha-condit}
\Vert  H_{\alpha}^{s}(F,G) \Vert_{p,s} \leq 
c_{p,q} \Vert \Gamma_{F,s}^{-1}  \, D F\Vert^k_{k, 2^{k-1}r,s}\Vert G\Vert^k_{k, q,s},
\end{equation}
 where $\frac1p=\frac1q+\frac1r$.
\end{proposition}

\smallskip

In order to get our bivariate density bound, we shall also need to work on weighted norms on the interval $[s,t]$. For instance, when $H>1/2$, we have the following uniform scale invariant inequalities:

\begin{lemma}\label{kjh} Assume $H>1/2$.
Let $0 <\varepsilon <1$ and $\gamma> H-\frac{1}{2}$. There exist two constants $C_1,C_2 >0$ such that for any continuous $f:[0,1]\to \mathbb{R}^{n}$,  and $\varepsilon \le  s<t \le 1$, we have:
\begin{equation}\label{eq:ineq-norm-H-in-s-t}
C_1 (t-s)^{2H} \frac{\min_{[0,1]} |f|^4}{\| f\|_{\infty}^2+ \| f \|_{\gamma}^2} 
\le \int_s^t \left|  \int_u^t \partial_v K(v,u) f(v)dv\right|^2 du %  \le C_2  (t-s)^{2H} \| f \|_\infty^2.
\end{equation}
\end{lemma}

\begin{proof}

For notational sake, we prove our lemma for real valued functions only, leaving the obvious extension to $f:[0,1]\to \mathbb{R}^{n}$ to the patient reader. We now proceed in several steps.

\smallskip

\noindent
\textit{Step 1:} We first prove that
\begin{align}\label{rich}
\alpha \int_s^t \left(  \int_u^t (v-u)^{H-3/2} f(v)dv\right)^2 du & \le \int_s^t \left(  \int_u^t \partial_v K(v,u) f(v)dv\right)^2 du   
% & \le \beta  \int_s^t \left(  \int_u^t (v-u)^{H-3/2} f(v)dv\right)^2 du.
\end{align}
Using the change of variable $u=s+sx$ and $v=s+sy$, and the scaling property of the kernel $K$, we just need to prove that for $t \le T$,
\begin{align*}
\alpha \int_0^t \left(  \int_u^t (v-u)^{H-3/2} f(v)dv\right)^2 du & \le \int_0^t \left(  \int_u^t \partial_v K(v+1,u+1) f(v)dv\right)^2 du  
% & \le \beta  \int_0^t \left(  \int_u^t (v-u)^{H-3/2} f(v)dv\right)^2 du.
\end{align*}
Up to a constant, the norm $ \int_0^t \left(  \int_u^t (v-u)^{H-3/2} f(v)dv\right)^2 du$ is the norm of the reproducing Hilbert space of the Gaussian process
\begin{equation}\label{eq:def-process-Y}
Y_t=d_H \int_0^t (t-s)^{H-1/2} dW_s,
\end{equation}
where $d_H (H-1/2)=c_H$ and the norm $\int_0^t \left(  \int_u^t \partial_v K(v+1,u+1) f(v)dv\right)^2 du$ is the norm of the reproducing Hilbert space of the Gaussian process
\[
Z_t=\int_0^t K(t+1,s+1) dW_s.
\]
So, to prove \eqref{rich}, according to Lemma 2 in \cite{BN1}, we just need to prove that $(Y_t)_{0 \le t \le T}$ and $(Z_t)_{0 \le t \le T}$ are equivalent in distribution. From Theorem 7 in \cite{BN1}, we have to prove that there exists a square integrable kernel $L$ such that
\[
K(t+1,s+1)=d_H (t-s)^{H-1/2} +d_H \int_s^t(t-r)^{H-1/2} L(r,s) dr.
\]
Since $H>1/2$, we can differentiate both members of the above equation with respect to $t$ and we obtain
\[
 (s+1)^{1/2 -H} (t-s)^{H-3/2}(t+1)^{H-1/2}- (t-s)^{H-3/2}=\int_s^t(t-r)^{H-3/2} L(r,s) dr.
\]
Hence, it suffices to take
\[
L(t,s)=\frac{1}{\Gamma\left( H-\frac{1}{2} \right)} D^{H-1/2}_{s+} \left[ (\cdot -s)^{H-3/2}\left( \left( \frac{\cdot+1 }{s+1 }\right)^{H-1/2}-1 \right) \right](t),
\]
which is easily seen to be square integrable.

\smallskip

\noindent
\textit{Step 2:} Thanks to the previous step, our result boils down to show \eqref{eq:ineq-norm-H-in-s-t} when $\partial_v K(v,u)$ is replaced by $(v-u)^{H-3/2}$. Towards this aim, notice first that by using the same arguments as in \cite[Lemma 4.4]{BH}, we obtain the interpolation inequality
\[
\int_0^1 \left(  \int_u^1 (v-u)^{H-3/2} f(v)dv\right)^2 du\ge C \frac{\min_{[0,1]} f^4}{\|f \|^2_{\infty}+\| f \|_{\gamma}^2} ,
\]
which is easy to rescale:
\begin{align*}
\int_s^t \left(  \int_u^t (v-u)^{H-3/2} f(v)dv\right)^2 du & =(t-s)^{2H}\int_0^1 \left(  \int_u^1 (v-u)^{H-3/2} f(s+(t-s)v)dv\right)^2 du \\
 & \ge C (t-s)^{2H} \frac{\min_{[0,1]} f_{st} ^4}{\|f_{st} \|^2_{\infty}+\| f_{st} \|_{\gamma}^2} \\
 &\ge C (t-s)^{2H} \frac{\min_{[0,1]} f^4}{\|f \|^2_{\infty}+\| f \|_{\gamma}^2},
\end{align*}
where we have set $f_{st}(u)=f(s+(t-s)u)$.
\end{proof}

In the case $H \le 1/2$, the scale invariant inequalities we have are the following:

\begin{lemma}\label{lem:interpol-st-less-1/2}
Assume $H \le 1/2$.
Let $0 <\varepsilon <1$ . There exists constants $c_1,c_2 >0$ such that for any $ f \in C^\gamma([0,1];\R^{n})$, with $ \gamma >1/2 -H$  and $\varepsilon \le  s<t \le 1$, we have:
\begin{align*}
c_1 (t-s)^{2H} \min_{[0,1]} |f|^2  &  \le 
\int_s^t \left|K(t,u)  f(u) + \int_u^t \lp  f( r ) -f(u) \rp  \, \partial_r K(r,u)  dr\right|^2 du  
% & \le c_2 (t-s)^{2H} (\| f\|_\gamma^2+\|f \|_\infty^2) .
\end{align*}
\end{lemma}

\begin{proof}
Some elements of the proof are pretty similar to the proof of Lemma \ref{kjh}, so we only sketch the main arguments. We also focus here on the case of real valued functions for sake of readability.
\smallskip

\noindent
\textit{Step 1}. Set $\hat{L}(t,s)=(t-s)^{H-1/2}$. Along the same line as for Lemma \ref{kjh}, it is readily checked that:
\begin{align*}
 &C  \int_s^t \left(\hat{L}(t,u)  f(u) + \int_u^t \lp  f( r ) -f(u) \rp  \, \partial_r \hat{L}(r,u)  dr\right)^2 du  \\
\le & \int_s^t \left(K(t,u)  f(u) + \int_u^t \lp  f( r ) -f(u) \rp  \, \partial_r K(r,u)  dr\right)^2 du,
\end{align*}
by using a scaling argument and the equivalence in distribution of the two processes $\int_0^t K(t+1,s+1)dW_s$ and $d_H \int_0^t (t-s)^{H-1/2} dW_s$.

\smallskip

\noindent
\textit{Step 2}.  Some fractional calculus arguments show that the following lower bound holds true:
\[
 \int_0^1 \left(\hat{L}(1,u)  f(u) + \int_u^1 \lp  f( r ) -f(u) \rp  \, \partial_r \hat{L}(r,u)  dr\right)^2 du \ge C \int_0^1 f(u)^2 du \ge C \min_{[0,1]} f^2 ,
\]
which can be rescaled to get the lower bound of our claim.
\end{proof}

\smallskip

As a last preliminary step before the proof of our bivariate density bound, let us mention that we shall express some of our Malliavin derivatives bounds in terms of H\"older norms on the interval $[s,t]$. However, it will be more convenient to work with Besov norms rather than H\"older's because Besov norms are smooth in the Malliavin calculus sense. This is why we introduce the following quantities: if $Y$ is a process which is $\gamma$-H\"older, $1/2<\gamma <H$, set
\[
\cn_{\gamma,p}^{s,t}(Y)=\int_{s}^{t}\int_{s}^{t}
\frac{|Y_v-Y_u|^{2p}}{|v-u|^{2\gamma p+2}}\,dudv,
\]
where $\gamma <H$ and $p>0$.
Then from the Besov-H\"older embedding we have
\[
\|Y\|_{s,t,\gamma} \le C\left( \cn_{\gamma,p}^{s,t}(Y)\right)^{1/2p}, \quad 0 \le s \le t \le 1.
\]
From the Garsia-Rodemich-Rumsey inequality in Carnot groups, this embedding extends to the rough paths case. More precisely,  if $Y$ is a $\gamma$-rough path with lift $\mathbf{Y}$, then,
\[
\| \mathbf{Y}\|_{s,t,\gamma} \le C\left( \cn_{\gamma,p}^{s,t}(\mathbf{Y})\right)^{1/2p}, \quad 0 \le s \le t \le 1,
\]
where now,
\[
\cn_{\gamma,p}^{s,t}(\mathbf{Y})=\int_{s}^{t}\int_{s}^{t}
\frac{|\Gamma_{u,v}|^{2p}}{|v-u|^{2\gamma p+2}}\,dudv.
\]
with 
\[
\Gamma_{s,t}=\sum_{k=1}^{[1/\gamma]} \left\|  \int_{\Delta^k [s,t]}  dY^{\otimes k} \right\|^{1/k}.
\]
With this notation in mind, using the interpolation inequalities we just proved and arguing as in Section 4 we obtain then the following estimates.

\begin{proposition} \label{estimate bivar mall} Let $\varepsilon \in (0,1)$, and consider $H\in(1/4,1)$. Recall that the Malliavin matrix $\Gamma_{F}$ of a random variable $F$ with derivatives taken with respect to the Wiener process $W$ are defined by \eqref{eq:def-conditional-sobolev}. Then there exist constants $C,r >0$ such that for $ \varepsilon \le s \le t \le 1$ the following bounds hold true for $\gamma <H$:
 \begin{eqnarray*}
 \Vert \Gamma_{X^x_{t}-X^x_{s},s}^{-1}   \Vert^n_{n, 2^{n+2},s} 
 &\le& 
 \frac{C}{(t-s)^{2nH}} \, \mathbb{E}_s^{\frac{n}{2^{n+2}}} \!\left[ (1+  \cn_{\gamma,p}^{0,1} (\mathbf{M}))^r  \right]  \\
 \Vert D(X^x_{t}-X^x_{s}) \Vert^n_{n, 2^{n+2},s} 
 &\le& 
 C  \, (t-s)^{nH} \mathbb{E}_s^{\frac{n}{2^{n+2}}} \!\left[ (1+  \cn_{\gamma,p}^{0,1} (\mathbf{M}) )^r  \right],
\end{eqnarray*}
where 
\[
M=(B, \hat{Y}, X^x, \mathbf{J}, \mathbf{J}^{-1}),
\]
with $\hat{Y}_t=\int_0^t (t-s)^{H-1/2} d\hat{W}_s$ where $\hat{W}$ is a Brownian motion independent from $W$.
\end{proposition}

\begin{proof}
Taking into account the interpolation inequalities of Lemmas \ref{kjh} and \ref{lem:interpol-st-less-1/2}, the bound
\[
 \Vert \Gamma_{X^x_{t}-X^x_{s},s}^{-1}   \Vert^n_{n, 2^{n+2},s}  \le \frac{C}{(t-s)^{2nH}} \, \mathbb{E}_s^{\frac{n}{2^{n+2}}} \!\left[ (1+  \cn_{\gamma,p}^{0,1} (\mathbf{M}))^r  \right] 
\]
 follows along the same lines as in Section \ref{sec:upper-bounds}. We now turn to the upper bound for the Malliavin derivative. Again, we use the method by Inahama \cite{Inahama}.
 Set
 \[
 \Theta_1(t)=\mathbf{J}_t \int_s^t K_t^*( \mathbf{J}^{-1} V(X))(u) d\hat{W}(u)
 \]
 where
  \begin{align*}
  K_t^* ( \mathbf{J}^{-1} V(X))(v) =
 \begin{cases}
 \int_v^t \partial_u K(u,s) \mathbf{J}^{-1}_{u} V_j(X_u) du, \quad H >1/2 \\
K(t,v)  \mathbf{J}^{-1}_{v} V_j(X_v) + \int_v^t  \lp  \mathbf{J}^{-1}_{r} V_j(X_r)- \mathbf{J}^{-1}_{s} V_j(X_s) \rp  \, \partial_r K(r,s)  dr, \quad H \le 1/2.
\end{cases}
\end{align*}
As in Inahama \cite{Inahama}, we have
\[
\| D(X^x_t-X^x_s)\|_{L^2_s} \le C \hat{\mathbb{E}} (| \Theta_1(t)|^2 )^{1/2}.
\]
From the previous lemmas, we can estimate
\[
 \hat{\mathbb{E}} (| \Theta_1(t)|^2 )^{1/2} \le C  \hat{\mathbb{E}} (| \tilde{\Theta}_1(t)|^2 )^{1/2},
\]
where
 \[
 \tilde{\Theta}_1(t)=\mathbf{J}_t \int_s^t \hat{L}_t^*( \mathbf{J} V(X))(u) d\hat{W}(u),
 \]
 with, as before, $\hat{L}(t,s)=(t-s)^{H-1/2}$. We can now write $\tilde{\Theta}$ as a rough paths integral,
 \[
 \tilde{\Theta}_1(t)=\mathbf{J}_t \int_s^t \mathbf{J}^{-1}_u V(X_u) d\hat{Z}(u),
 \]
 where
 \[
 \hat{Z}(u)=\int_s^u  (s-v)^{H-1/2}d\hat{W}(v).
 \]
 The advantage of working with the kernel $(s-v)^{H-1/2}$ is that it is translation invariant, so it is easily seen that we have in distribution with respect to $\hat{\mathbb{P}}$ (that is $W$ is fixed),
 \[
 \tilde{\Theta}_1(t)=(t-s)^{H} \mathbf{J}_t \int_0^1 \mathbf{J}^{-1}_{s+(t-s)u}  V(X_{s+(t-s)u}) d\hat{Y}(u),
 \]
 where $\hat{Y}$ is an independent copy of the process $Y$ defined by \eqref{eq:def-process-Y}.
 Using rough paths theory, as in Section 4, we get an upper bound of the form $(1+  \cn_{\gamma,p}^{0,1} (\mathbf{M}) )^r$ for the integral $\int_0^1 \mathbf{J}^{-1}_{s+(t-s)u}  V(X_{s+(t-s)u}) d\hat{Y}(u)$. Thus we get
 \begin{align*}
\mathbb{E}_s\left( \| D(X^x_t-X^x_s)\|^n_{L^2_s}\right)^{1/n} & \le C (t-s)^{H} \mathbb{E}_s \left( \hat{\mathbb{E}} \left(  (1+  \cn_{\gamma,p}^{0,1} (\mathbf{M}) )^{2r}\right)^{n/2} \right)^{1/n} \\
 & \le  C (t-s)^{H} \mathbb{E}_s \left(  \left(  1+  \cn_{\gamma,p}^{0,1} (\mathbf{M}) \right)^{rn}\right)^{1/n}.
 \end{align*}
 Higher order derivatives are treated similarly.
\end{proof}

We are finally  ready for the proof of Condition {\bf (A2)}.

\begin{proof}[Proof that Condition {\bf (A2)} holds with $\beta=n$]

In all the proof the range  of the parameters $s,t$ will be $\varepsilon <  s \le t \le 1$ where $0<\varepsilon <1$. Also $C$ will denote a deterministic constant  that varies from line to line but which is independent from $s,t$ (however it may depend on other parameters like $n,p,V_i,\varepsilon$). 

Consider the joint probability density function of the $2n$-dimensional random vector $(X^x_t,$ $X^x_s)$ with $s<t$ denoted $p_{s,t}(z_1, z_2)$ (the fact that it exists as a smooth function is a consequence of Proposition \ref{estimate bivar mall}).
We then write
$$
p_{s,t}(z_1, z_2)=\hat{p}_{s, t-s}(z_1, z_2-z_1), \quad \text{ for } z_1, z_2 \in \R^n,
$$
where $\hat{p}_{s, t-s}(\cdot, \cdot)$ denotes the density of the random vector $(X^x_s, X^x_t-X^x_s)$. We now bound the function $\hat{p}_{s, t-s}$, which shall be expressed as
\begin{align*}
\hat{p}_{s, t-s}(\xi_{1},\xi_{2})
& = \EE\lc \delta_{\xi_1}(X_s^x) \, \delta_{\xi_2}(X_t^x-X_s^x) \,  \rc,
\quad\text{for}\quad
\xi_{1},\xi_{2}\in\R^{n},\\
 & = \EE\lc \delta_{\xi_1}(X_s^x) \, \EE_{s}\lc \delta_{\xi_2}(X_t^x-X_s^x)  \rc \rc.
\end{align*}
The idea is now to bound $M_{st}= \EE_{s}\lc \delta_{\xi_2}(X_t^x-X_s^x)  \rc $ by using first the conditional integration by parts formula in Proposition \ref{prop:int-parts-cond-W} and then Cauchy-Schwarz inequality. We obtain
\begin{equation} \label{term0}
|M_{s,t}| \le  C \Vert \Gamma_{X^x_{t}-X^x_{s},s}^{-1}   \Vert^n_{n, 2^{n+2},s} \, \Vert D(X^x_{t}-X^x_{s}) \Vert^n_{n, 2^{n+2},s}  \EE_{s}^{1/2}\lc \1_{(X_t^x-X_s^x > \xi_{2})}  \rc.
\end{equation}
Thus, owing to Proposition \ref{estimate bivar mall} we obtain:
\begin{equation}\label{eq:first-bnd-bivariate}
\hat{p}_{s, t-s}(\xi_{1},\xi_{2}) \le  
\frac{C}{(t-s)^{nH}}   
\EE\lc \delta_{\xi_1}(X_s^x) \,
\mathbb{E}_s^{\frac{n}{2^{n+1}}}\! \left[ (1+  \cn_{\gamma,p}^{0,1} (\mathbf{M}))^r  \right] \,
\EE_{s}^{1/2}\!\left[ \1_{(X_t^x-X_s^x > \xi_{2})} \right] \rc
\end{equation}
Furthermore, it is readily checked that
\begin{equation*}
|X_t^x-X_s^x| \le C\, |t-s|^\gamma 
\cn_{\ga,2p}^{1/2p}(\mathbf{M}),
\end{equation*}
and thus, for $q$ arbitrarily large, we have
\begin{equation*}%\label{eq:bnd-cdt-probab}
\EE_{s}\left[ \1_{(X_t^x-X_s^x > \xi_{2})} \right] \le 
 C \, \left( 1 \wedge \frac{|t-s|^{\gamma q}}{\xi^q_2} \EE_{s}\left[ \cn_{\ga,2p}^{q}(\mathbf{M})\right] \right).
\end{equation*}
Plugging this inequality into \eqref{eq:first-bnd-bivariate}, we end up with:
\begin{equation}\label{eq:bnd-density-conditional}
\hat{p}_{s, t-s}(\xi_{1},\xi_{2}) \le  \frac{C}{(t-s)^{nH}}   \,
\EE\left[ \delta_{\xi_1}(X_s^x) \, \Psi_1 \, 
\left( 1 \wedge \frac{|t-s|^{\gamma q} }{\xi^q_2} \, \Psi_2 \right) \right]
\end{equation}
where $\Psi_1$ and $\Psi_2$ are two random variables which are smooth in the Malliavin calculus sense. We can now integrate \eqref{eq:bnd-density-conditional} safely by parts in order to regularize the term $\delta_{\xi_1}(X_s^x)$, which finishes the proof.

\end{proof}

\subsection{Lower bound on hitting probabilities}

We now apply Theorem \ref{thm:cap:LB}, which yields the lower bound of Theorem \ref{thm:hitting-capacity-X}.

\begin{theorem}\label{thm:cap:LB2}
 Let $X_t^x$ denote the solution to equation \textnormal{(\ref{eq:sde})} where $B$ is a fractional Brownian motion with Hurst parameter
$H >\frac14$ and where the vector fields $V_1,\ldots,V_d$ satisfy Hypothesis \ref{hyp:elliptic}. Fix $0<a<b\le 1$ and $M >0$. Then there exists a positive constant $c=c(a,b,H,M,n)$ such that for all compact sets $A\subseteq[-M\,,M]^n$,
            \begin{equation*}
                \mathbb{P}(X_t^x([a,b]) \cap A
                \neq\varnothing) \geq c\, 
                \textnormal{Cap}_{n-\frac{1}{H}}(A).
            \end{equation*}
\end{theorem}

\begin{proof}
Since we have already proved that Hypothesis {\bf (A2)} holds with $\beta=n$, it suffices to verify Hypotheses {\bf (A1)} 
of Theorem \ref{thm:cap:LB}. 
First of all, observe that, owing to Theorem \ref{thm:strict-positivity-intro}, the density of our process $p_t(y)$ is strictly positive and continuous in $y$. Moreover, our results of Section~5 also show that this density is uniformly bounded for $t \in [a,b]$ and $y \in \R^n$. Therefore, it holds that for all $z \in [-M,M]^n$,
$$
\int_a^b p_{t}(z) dt \geq \inf_{ \vert z \vert \leq M} \int_a^b p_{t}(z) dt =C(a,b,M)>0,
$$
which proves that {\bf (A1)} holds true.
\end{proof}

As a consequence of Theorem \ref{thm:cap:LB2} and Corollary \ref{cor11}, we have the following result
on hitting points for the process $X_t^x$.
\begin{corollary}
Under the hypotheses of Theorem \ref{thm:cap:LB2}, if $n <\frac{1}{H}$, the process $X_t^x$ hits points
in $\R^n$ with positive probability.
\end{corollary}

\subsection{Upper bounds on hitting probabilities}

As in the last subsection, we provide a general result that gives sufficient conditions on a continuous stochastic process in order to obtain an upper bound for the hitting probabilities of the process in terms of the Hausdorff measure. The proof follows along the same lines as in  \cite[Theorem 3.1]{DKN}, but for the sake of completeness we sketch the main steps.

Given $\alpha\geq 0$, the $\alpha$-dimensional Hausdorff measure
of a set $A$ in $\R^n$ is defined as
\begin{equation}\label{eq:HausdorfMeasure}
	{\mathcal{H}}_ \alpha (A)= \lim_{\epsilon \rightarrow 0^+} \inf
	\left\{ \sum_{i=1}^{\infty} (2r_i)^ \alpha : A \subseteq
	\bigcup_{i=1}^{\infty} B(x_i, r_i), \ \sup_{i\ge 1} r_i \leq
	\epsilon \right\},
\end{equation}
where $B(x,r)$ denotes the open (Euclidean) ball of radius
$r>0$ centered at $x\in \R^n$. When $\alpha <0$, we define 
$\mathcal{H}_\alpha (A)$ to be infinite.

Let us now consider a continuous stochastic process $(u_{t}, t\geq 0)$ in $\R^n$, and
for all positive integers $N$ and $H \in(0,1)$, set $t_k^{N,H}:=k 2^{-\frac{N}{H}}$,
and $I^{N,H}_k = [t_k^{N,H},t_{k+1}^{N,H}]$.
\begin{theorem}\label{thm:meas:UB} 
    Fix $0<a<b$, $\beta>0$, and $M>0$.
Suppose that there exists $H \in (0,1)$ and $c_H >0$ such that for
	all $z \in [-M,M]^n$, $\epsilon>0$, large $N$ and $I^{N,H}_k \subseteq [a,b]$, 
	\begin{equation}\label{cond:i}
		\mathbb{P}( u(I^{N,H}_k) \cap B(z,\epsilon) \neq \varnothing) \leq c_H\, \epsilon^{\beta}.
	\end{equation}
    Then there exists a positive constant $C=C(a,b,\beta,M,H,n)$ such that for all Borel sets $A \subset [-M,M]^n$,
\begin{equation*}
                \mathbb{P}(u([a,b]) \cap A
                \neq\varnothing) \leq C\mathcal{H}_{\beta-\frac{1}{H}}(A).
            \end{equation*}
  \end{theorem}
  
  \begin{Remark}
Because of the inequalities between capacity and Hausdorff measure, the right-hand side of Theorem \ref{thm:meas:UB} 
can be replaced by $C \, \textnormal{Cap}_{\beta-\frac{1}{H}-\epsilon}(A)$,  (cf. \cite[p. 133]{Kahane:85}).
\end{Remark}
  
  \begin{proof}
When $\beta < \frac{1}{H}$, there is nothing to prove, so we assume
that $\beta-  \frac{1}{H}>0$. Fix $\epsilon \in (0,1)$ and $N \in \mathbb{N}$ such that $2^{-N-1}<\epsilon\leq 2^{-N}$, and write
$$
\mathbb{P}(u([a,b]) \cap B(z,\epsilon)
                \neq\varnothing) \leq \sum_{k:I^{N,H}_k \cap [a,b] \neq\varnothing} \mathbb{P}(u(I^{N,H}_k) \cap B(z,\epsilon)
                \neq\varnothing),
$$
where the number of $k$'s involved in the sum is at most 
$2^{\frac{N}{H}}$. Then, hypothesis (\ref{cond:i}) implies that
for all large $N$ and $z \in A$,
$$
\mathbb{P}(u([a,b]) \cap B(z,\epsilon)
                \neq\varnothing) \leq \tilde{C} 2^{-N(\beta-\frac{1}{H})}
                \leq C \epsilon^{\beta-\frac{1}{H}}.
$$
Finally, a covering argument concludes the desired proof.
  \end{proof}
  
  By the definition of Hausdorff measure and as a consequence of Theorem \ref{thm:meas:UB}, we have the following result
on hitting points for the process $u$.
\begin{corollary} \label{cor:h}
Under the hypotheses of Theorem \ref{thm:meas:UB}, if $\beta >\frac{1}{H}$, the process $u$ does not hit points 
in $\R^n$ a.s., that is,
$$
\mathbb{P}( \exists \, t > 0: u_{t}=x)=0, \quad \text{for all } x \in \R^n.
$$
\end{corollary}

\begin{proof}
If $\beta >\frac{1}{H}$, then $\mathcal{H}_{\beta-\frac{1}{H}}(\{ x \})=0$ by the definition of Hausdorff measure,
and the result follows from Theorem \ref{thm:meas:UB}.
\end{proof}

The next result provides sufficient conditions that imply Hypothesis (\ref{cond:i}) of Theorem~\ref{thm:meas:UB}. These conditions are
easier to verify for non-linear equations than Hypothesis~(\ref{cond:i}). The proof follows exactly as the proof of \cite[Theorem 3.3]{DKN}, and is therefore ommitted. It suffices to replace the parabolic metric 
$\Delta((t,x);(s,y))=\vert t-s \vert^{1/2}+\vert x-y \vert$ therein by our fractional metric $\vert t-s \vert^{2H}$.
\begin{theorem} \label{cond}
	Fix $0<a<b$ and $M>0$. Assume that the $\R^n$-valued
	stochastic process $u$ satisfies the following two conditions:
	\begin{itemize}
		\item[\textnormal{(i)}] For any $t>0$,
			the random vector $u_{t}$ has a density $p_{t}(z)$ which
			is uniformly bounded over $z \in [-M,M]^n$ and $t \in [a,b]$.

		\item[\textnormal{(ii)}] For some $H \in (0,1)$ and for all $p>1$, there exists a 
			constant $C=C(p,H,a,b)$ such that for
			any $s,t \in [a,b]$, 
			\begin{equation*} 
				\me [ \vert u_{t}-u_{s} \vert^p] \leq C \vert t-s \vert^{Hp}.
			\end{equation*}
	\end{itemize}
Then for any $\beta \in\,]0\,,n[$, Condition 
    \textnormal{(\ref{cond:i})} in Theorem \ref{thm:meas:UB} is satisfied for such $\beta$.
    \end{theorem}

Let us now apply this general theory to the $n$-dimensional process solution to  equation \textnormal{(\ref{eq:sde})}.
\begin{theorem}\label{thm:haus}
 Let $X_t^x$ denote the solution to equation \textnormal{(\ref{eq:sde})} where $B$ is a fractional Brownian motion with Hurst parameter
$H >\frac14$ and the vector fields satisfy Hypothesis \ref{hyp:elliptic}. Fix $0<a<b \leq 1$, $M >0$ and $\eta>0$. Then there exists a positive constant $C=C(a,b,H,M,n, \eta)$ such that for all Borel sets $A\subseteq[-M\,,M]^n$,
            \begin{equation*}
                \mathbb{P}(X_t^x([a,b]) \cap A
                \neq\varnothing) \leq C\, \mathcal{H}_{n-\frac{1}{H}-\eta}(A).
            \end{equation*}
\end{theorem}

\begin{Remark}
Because of the inequalities between capacity and Hausdorff measure, the right-hand side of Theorem \ref{thm:meas:UB} 
can be replaced by $C \, \textnormal{Cap}_{n-\frac{1}{H}-\eta'}(A)$,  (cf. \cite[p. 133]{Kahane:85}).
\end{Remark}

As a consequence of Theorem \ref{thm:haus} and Corollary \ref{cor:h}, we have the following result
on hitting points for the process $X_t^x$.
\begin{corollary}
Under the hypotheses of Theorem \ref{thm:haus}, if $n >\frac{1}{H}$, the process $X_t^x$ does not hit points 
in $\R^n$ a.s.
\end{corollary}

\noindent {\it Proof of Theorem \ref{thm:haus}.}
It suffices to check that Conditions (i) and (ii) of Theorem \ref{cond} 
hold true for the solution to our equation \textnormal{(\ref{eq:sde})}.
Condition (i) follows straightforwardly from our results in Section \ref{sec:upper-bounds}.
Condition (ii) follows from \eqref{eq:bnd-X-pvar}.
\qed

\end{document}